\documentclass{tran-l}

\usepackage{amsfonts}
\usepackage{amsmath}
\usepackage{amsfonts}
\usepackage{amssymb}
\usepackage{amscd}

\newtheorem{theorem}{Theorem}[section]
\newtheorem{lemma}[theorem]{Lemma}
\newtheorem{corollary}[theorem]{Corollary}
\newtheorem{prop}[theorem]{Proposition}

\theoremstyle{definition}
\newtheorem{definition}[theorem]{Definition}

\theoremstyle{remark}
\newtheorem{remark}[theorem]{Remark}

\numberwithin{equation}{section}

\newcommand{\erre}{\mathbb{R}}
\newcommand{\p}{\mathbb{P}}
\newcommand{\hess}{\operatorname{hess}}
\newcommand{\Hess}{\operatorname{Hess}}
\newcommand{\pair}[1]{\langle#1\rangle}
\newcommand{\ra}{\rightarrow}
\newcommand{\esse}{\mathbb{S}}
\newcommand{\di}{\mathrm{d}}
\newcommand{\norm}[1]{{\|#1\|}}
\newcommand{\enne}{\mathbb{N}}
\newcommand{\sgn}{\operatorname{sgn}}
\newcommand{\R}{\operatorname{R}}
\newcommand{\pmin}{p_\mathrm{min}}
\newcommand{\pmax}{p_\mathrm{max}}
\newcommand{\ricc}{\operatorname{Ricc}}

\begin{document}

\title[Hypersurfaces in warped products]{Hypersurfaces of constant higher order mean curvature in warped products}



\author[L. J. Al\'ias]{Luis J. Al\'ias}
\address{Departamento de Matem\'aticas, Universidad de Murcia, Campus de Espinardo, 30100 Espinardo, Murcia, Spain.}
\email{ljalias@um.es}
\thanks{L.J. Al\'ias was partially supported by MICINN project MTM2009-10418 and Fundaci\'{o}n S\'{e}neca
project 04540/GERM/06, Spain. This research is a result of the activity developed within the framework of the Programme in Support of Excellence Groups of the Regi\'{o}n de Murcia, Spain, by Fundaci\'{o}n S\'{e}neca, Regional Agency
for Science and Technology (Regional Plan for Science and Technology 2007-2010).}

\author[D. Impera]{Debora Impera}
\address{Dipartimento di Matematica,
Universit\`{a} degli studi di Milano, via Saldini 50, I-20133 Milano, Italy.} \email{debora.impera@unimi.it}

\author[M. Rigoli]{Marco Rigoli}
\address{Dipartimento di Matematica,
Universit\`{a} degli studi di Milano, via Saldini 50, I-20133 Milano, Italy.} \email{marco.rigoli@unimi.it}
\thanks{M. Rigoli was partially supported by MEC Grant SAB2010-0073}


\date{January 11, 2011; revised December 15, 2011}

\dedicatory{}

\begin{abstract}
In this paper we characterize compact and complete hypersurfaces with some constant higher order mean curvature into warped product spaces. Our approach is based on the use of a new trace operator version of the Omori-Yau maximum
principle which seems to be interesting in its own.
\end{abstract}

\maketitle

\section{Introduction}

A classical result by Alexandrov \cite{al} states that a compact hypersurface with constant mean curvature embedded in Euclidean space must be a round sphere. The original proof is based on a clever use of the maximum
principle for elliptic
partial differential equations. This method, now called the Alexandrov's reflexion
method, also works for hypersurfaces in ambient spaces having a sufficiently large number of isometric reflexions, for instance in the hyperbolic space.

To extend the above result to a larger class of Riemannian spaces it appears convenient to
consider manifolds with a sufficiently large family of complete embedded constant mean
curvature hypersurfaces. Such a family plays the role of the umbilical hypersurfaces in spaces of constant sectional curvature. In this setting, given an immersed hypersurface, the next step is to look for geometric assumptions that force the hypersurface to be one of the
selected family. In the compact case, this was first done by Montiel \cite{montiel} that considers as a natural class of ambient manifolds that of warped products $M^{n+1}= \erre\times_\rho\p^n$ where $\p^n$ is a complete
$n$-dimensional Riemannian manifold and $\rho:\erre\rightarrow \erre_+$ is a smooth warping function. Then each leaf $\p_t=\{t\}\times\p^n$ (called here a
{\it slice\/}) of the foliation
$t\in \erre\mapsto \p_t$ of $M^{n+1}$ is a complete hypersurface with constant mean
curvature. This approach was later considered in \cite{aliasdajczer} where Al\'ias and Dajczer generalized Montiel's results. Some of those generalizations hold even for complete, not necessarily compact, hypersurfaces.\\
The aim of the present paper is to extend the investigation to hypersurfaces with constant higher order mean curvatures, both in the compact and in the complete case. Our main analytical tools to reach the goal are provided by the Newton tensors $P_k$ of the hypersurface, their associated second order differential operators $L_k$ and further various combinations of them. We underline that in the complete case, we tailor an appropriate version of the Omori-Yau maximum principle for very general operators to deal with the problems at hand.\\
The paper begins with Section 2, collecting a number of preliminary results and fixing notation. Section 3 is devoted to a proof of a generalized version of the Omori-Yau maximum  principle for a wide class of trace operators and to determine some geometrical assumptions guaranteeing the validity of the principle (see for instance Corollary \ref{OYproperly}). In Section 4, as a first application of our method and inspired by the mean curvature estimates given in \cite{aliasdajczer2}, we derive higher order mean curvature estimates for complete immersed hypersurfaces. In Sections 5 and 6 we determine sufficient conditions for hypersurfaces with constant higher order mean curvatures contained in a slab to be a slice of the ambient space, extending previous results for the case of constant mean curvature hypersurfaces given in \cite{aliasdajczer}. Finally, in Section 7 we give a number of further results recovering this uniqueness property.

\section{Preliminaries}
Let $f:\Sigma^n \rightarrow M^{n+1}$ be a connected hypersurface isometrically immersed into the Riemannian manifold
$M^{n+1}$. We let $A$ denote the second fundamental form of the immersion with respect to a (locally defined) normal
vector field $N$. Its eigenvalues,
$\kappa_1,\ldots,\kappa_n$, are the principal curvatures of the hypersurface (in the direction of $N$). Their elementary symmetric functions $S_k$, $k=0,...,n$,
$S_0=1$, define the $k$-mean curvatures of the immersion via the formula
$$
H_k= {n \choose k}^{-1}S_k.
$$
Thus $H_1=H$ is the mean curvature, $H_n$ is the Gauss-Kronecker curvature and $H_2$ is, when the ambient space is
Einstein, a multiple of the scalar curvature modulo a constant.

The Newton tensors associated to the immersion are inductively defined by
$$
P_0=I, \qquad P_k=S_k I-AP_{k-1}.
$$
Note, for further use, that $\mathrm{Tr}P_k=(n-k)S_k$ and $\mathrm{Tr}AP_k=(k+1)S_{k+1}$.
In the sequel we shall need to have the operators $P_k$ to be globally defined on $T\Sigma$. Obviously, the sign of the second fundamental form
$A$ depends on the chosen local unit field $N$.  However, when $k$ is even the sign of $S_k$ (and hence $H_k$) does not
depend on the chosen $N$, which implies, by its very definition, that the operator $P_k$ is a globally defined tensor field on $T\Sigma$. On the
other hand, when $k$ is odd in order to have $P_k$ globally defined, we need to assume that $\Sigma$ is
\textit{two-sided}. Recall that a hypersurface $f:\Sigma^n \rightarrow M^{n+1}$ is called two-sided if its normal bundle is trivial, i.e. there exists a
globally defined unit normal vector field $N$. For instance, every hypersurface with never vanishing mean curvature is trivially
two-sided. When the hypersurface is two-sided, a choice of $N$ on $\Sigma$ makes the second fundamental form $A$ and its associated Newton tensors $P_k$ globally defined tensor fields
on $T\Sigma$.

Let $\nabla$ stand for the Levi-Civita connection of $\Sigma$. For a given function
$u\in C^{2}(\Sigma)$,
we denote by $\hess{u}:T\Sigma\rightarrow T\Sigma$ the symmetric operator given by $\hess{u}(X)=\nabla_X\nabla u$ for every $X\in T\Sigma$, and by
$\Hess{u}:T\Sigma\times T\Sigma\rightarrow C^{\infty}(\Sigma)$ the metrically equivalent bilinear form given by
\[
\Hess{u}(X,Y)=\pair{\hess{u}(X),Y}.
\]
Associated to each globally defined Newton tensor $P_k:T\Sigma\ra T\Sigma$, we may consider the second order
differential operator $L_k:\mathcal{C}^{\infty}(\Sigma)\rightarrow\mathcal{C}^{\infty}(\Sigma)$ given by
$L_k=\mathrm{Tr}(P_k \circ \hess)$. In particular, $L_0$ is the Laplace-Beltrami operator $\Delta$. Observe that
\[
L_k(u)=\mathrm{div}(P_k\nabla u)-\pair{\mathrm{div}P_k,\nabla u},
\]
where $\mathrm{div}P_k=\mathrm{Tr}\nabla P_k$. This implies that $L_k$ is elliptic if and only if $P_k$ is positive definite and in this case the maximum principle holds for $L_k$. See for instance Theorem 3.1 in \cite{GT}.

Note that the ellipticity of the operator $L_1$ is guaranteed by the assumption $H_2>0$. Indeed, if this happens the mean curvature does not vanish on $\Sigma$, because of the basic inequality $H_1^2\geq H_2$. Therefore,
the immersion is two-sided and we can choose the normal unit vector $N$ on $\Sigma$ so that $H_1>0$. Furthermore
\[
n^2H_1^2=\sum_{j=1}^n\kappa_j^2+n(n-1)H_2>\kappa_i^2
\]
for every $i=1,\ldots, n$, and then the eigenvalues of $P_1$ satisfy $\mu_{1,i}=nH_1-\kappa_i>0$ for every $i$
(see, for instance, Lemma 3.10 in \cite{elbert}). This shows ellipticity of $L_1$. Regarding the ellipticity of $L_j$ when
$j\geq 2$, we will assume that there
exists an elliptic point in $\Sigma$, that is, a point $p\in\Sigma$ at which the second fundamental form $A$
is positive definite with respect to an appropriate orientation. The existence of an elliptic point implies that $H_k$ is positive at that point, and applying Garding
inequalities, \cite{Ga}, we have
\begin{equation}
\label{garding}
H_1\geq H_2^{1/2}\geq\cdots\geq H_{k-1}^{1/(k-1)}\geq H_k^{1/k}>0,
\end{equation}
with equality at any stage only for an umbilical point. Therefore, in case $H_k$ is constant, the immersion is two-sided and $H_1>0$ for the chosen
orientation.
Moreover, in this case, for every $1\leq j\leq k-1$, the operators $L_{j}$ are elliptic or, equivalently, the
operators $P_j$ are positive definite (for a proof of this fact see \cite[Proposition 3.2]{barbosacolares}). Observe that the existence of an elliptic point is not guaranteed, in general, even in the compact case. For instance,  it is clear that totally geodesic spheres and Clifford tori in $\esse^{n+1}$ are examples of compact isoparametric hypersurfaces without elliptic points. On the contrary, it is not difficult to see that every compact hypersursurface in an open hemisphere has elliptic points (see for instance the proof of Theorem 11.1 in \cite{aliasliramalacarne}).

In what follows, we consider the case when the ambient space is a warped product $M^{n+1}=I\times_{\rho} \p^n$,
where $I\subseteq\erre$ is an open interval, $\p^n$ is a complete $n$-dimensional Riemannian manifold and
$\rho:I \ra \erre_{+}$ is a smooth function.
The product manifold $I\times \p^n$ is endowed with the Riemannian metric
$$
\pair{,}=\pi_{I}^{*}(\di t^2)+\rho^2(\pi_{I})\pi_{\p}^*(\pair{,}_{\p}).
$$
Here $\pi_{I}$ and $\pi_{\p}$ denote the projections onto the corresponding factor and $\pair{,}_{\p}$ is the Riemannian
metric on $\p^n$. In particular, $M^{n+1}=I\times_{\rho} \p^n$ is complete if and only if $I=\erre$. We also observe that
each leaf $\p_t=\left\{t\right\} \times \p^n$ of the foliation $ t \ra \p_t$ of $M^{n+1}$ is a complete totally umbilical
hypersurface with constant $k$-mean curvature
$$
\mathcal{H}_k(t)=\Big( \frac{\rho'(t)}{\rho(t)}\Big)^k,\qquad 0 \leq k \leq n,
$$
with respect to $-\partial/\partial t$.

Let $f:\Sigma^n \ra M^{n+1}=I \times_{\rho} \p^n$ be an isometrically immersed hypersurface.
We define the \textsl{height function}  $h \in C^{\infty}(\Sigma)$ by setting $h=\pi_{I} \circ f$. In this context and following the terminology
introduced in \cite{aliasdajczer2}, we will say the the hypersurface is \textsl{contained in a slab} if $f(\Sigma)$ lies between two leaves $\p_{t_1}, \p_{t_2}$ with $t_1<t_2$ of
the foliation.\\
We observe that results similar to those of the present paper hold for spacelike hypersurfaces in a generalized Robertson-Walker spacetime.
These will appear in our paper \cite{aliasimperarigolilor}.

\section{The generalized Omori-Yau maximum principle for trace operators}

Let $\Sigma^n$ be an $n$-dimensional Riemannian manifold. Following the terminology introduced in \cite{pirise}, the
\textit{Omori-Yau maximum principle} is said to hold on $\Sigma$ for the Laplace operator if, for any
smooth function $u\in\mathcal{C}^2(\Sigma)$ with $u^*=\sup_\Sigma u<+\infty$ there exists a sequence of points
$\{p_j\}_{j\in\mathbb{N}}$ in $\Sigma$ with the properties
\[
\textrm{(i)} \,\,\, u(p_j)>u^*-\frac{1}{j}, \,\,\, \textrm{(ii)} \,\,\,
\norm{\nabla u(p_j)}<\frac{1}{j},
\textrm{ and } \textrm{(iii)} \,\,\, \Delta u(p_j)<\frac{1}{j}.
\]
Equivalently, for any $u\in\mathcal{C}^2(\Sigma)$ with $u_*=\inf_\Sigma u>-\infty$ there
exists a sequence $\{p_j\}_{j\in\mathbb{N}}$ in $\Sigma$ satisfying
\[
\text{(i)} \,\,\, u(p_j)<u_*+\frac{1}{j}, \,\,\, \text{(ii)} \,\,\, \norm{\nabla u(p_j)}<\frac{1}{j},
\text{ and } \text{(iii)} \,\,\, \Delta u(p_j)>-\frac{1}{j}.
\]
In this sense, the classical result given by Omori \cite{Om} and Yau \cite{Y} states that
the Omori-Yau maximum principle holds on every complete Riemannian manifold with Ricci curvature bounded from below.
More generally, as shown by Pigola, Rigoli and Setti \cite[Example 1.13]{pirise}, a sufficiently controlled decay of the
radial Ricci curvature of the form
\[
\mathrm{Ric}_\Sigma(\nabla r,\nabla r)\geq-C^2G(r)
\]
where $r$ is the distance function on $\Sigma$ to a fixed point, $C$ a positive constant, and
$G:[0,+\infty)\rightarrow\mathbb{R}$ is a smooth function satisfying
\[
\textrm{(i)} \,\,\, G(0)>0, \,\,\, \textrm{(ii)} \,\,\, G'(t)\geq 0,  \,\,\, \textrm{(iii)} \,\,\,
\int_0^{+\infty}1/\sqrt{G(t)}=+\infty \textrm{ and }
\]
\[
\textrm{(iv)} \,\,\, \limsup_{t\rightarrow+\infty}tG(\sqrt{t})/G(t)<+\infty,
\]
suffices to imply the validity of the Omori-Yau maximum principle.

On the other hand, as observed again in \cite{pirise}, the validity
of Omori-Yau maximum principle on $\Sigma^n$ does not depend on curvature bounds  as much as one would expect. For
instance, the Omori-Yau maximum principle holds on every Riemannian manifold admitting a non-negative $C^2$ function
$\gamma$ satisfying the following requirements: (i) $\gamma(p)\rightarrow +\infty$ as $p\rightarrow \infty$;
(ii) there exists $A>0$ such that $\norm{\nabla\gamma}\leq A\sqrt{\gamma}$ off a compact set; and (iii) there exists
$B>0$ such that $\Delta\gamma\leq B\sqrt{\gamma}\sqrt{G(\sqrt{\gamma})}$ off a compact set, where $G$ is as above
(see \cite[Theorem 1.9]{pirise}).

For the proof of our main results in this paper, we will use the following generalization of \cite[Theorem 1.9]{pirise} for
trace type differential operators which includes the operators $L_k$.
\begin{theorem}\label{maxprinc}
Let $(\Sigma,\pair{,})$ be a Riemannian manifold and let $L=\mathrm{Tr}(P \circ \hess)$ be a semi-elliptic operator, where
$P:T\Sigma\rightarrow T\Sigma$ is a positive semi-definite symmetric tensor satisfying $\sup_\Sigma\mathrm{Tr}P<+\infty$.
Assume the existence of a non-negative $C^2$ function $\gamma$ with the properties
\begin{eqnarray}
\label{gamma1} &\gamma(p) \ra +\infty \qquad &\text{as } p \ra \infty,\\
\label{gamma2} &\exists A>0 \qquad &\text{such that }\norm{\nabla \gamma}\leq A\gamma^{\frac{1}{2}}\qquad \text{off a compact set,}\\
\label{gamma3} &\exists B>0 \qquad &\text{such that } L \gamma \leq B\gamma^{\frac{1}{2}}G(\gamma^{\frac{1}{2}})^{\frac{1}{2}}\qquad \text{off a compact set,}
\end{eqnarray}
where $G$ is a smooth function on $[0,+\infty)$ such that:
\begin{equation}\label{condG}
\begin{array}{ll}
\mathrm{(i)}\  G(0)>0, & \mathrm{(ii)}\  G'(t)\geq 0 \qquad \text{on } [0,+\infty),\\
\mathrm{(iii)}\  G(t)^{-\frac{1}{2}}\not \in L^1(+\infty),& \mathrm{(iv)}\  \limsup_{t \ra \infty} \frac{t G(t^{\frac{1}{2}})}{G(t)}<+\infty.
\end{array}
\end{equation}
Then, given any function $u \in C^{2}(\Sigma)$ with $u^*=\sup_{\Sigma}u < +\infty$, there exists a sequence $\left\{p_{j}\right\}_{j\in \enne} \subset \Sigma$ with the properties
\begin{equation}\label{omoriyau}
\mathrm{(i)}\ u(p_j)>u^*-\frac{1}{j},\ \mathrm{(ii)} \ \norm{\nabla u(p_j)}<\frac{1}{j}, \ \mathrm{(iii)} \ Lu(p_j)< \frac{1}{j}, \ \forall j\in\enne.
\end{equation}
\end{theorem}

\begin{proof}
Define the function
$$\varphi(t)=e^{\int_0^tG(s)^{-\frac{1}{2}}}\di s.
$$
Note that $\varphi(t)$ is a well defined, smooth, positive function such that
$\varphi(t) \ra +\infty$ as $t \ra +\infty$.
Moreover
$$
\varphi'(t)=G(t)^{-\frac{1}{2}}\varphi(t) \quad \mathrm{and} \quad \varphi''(t)=\left(G(t)^{-1}-2G(t)^{-\frac{3}{2}}G'(t)\right)\varphi(t),
$$
and therefore
\begin{equation}
\label{luis.1}
\Big(\frac{\varphi'(t)}{\varphi(t)}\Big)^2-\frac{\varphi''(t)}{\varphi(t)}=2G(t)^{-\frac{3}{2}}G'(t)\geq 0.
\end{equation}
Then, using assumption \eqref{condG},(iv) we get
\begin{equation}
\label{luis.2}
\frac{\varphi'(t)}{\varphi(t)} \leq c(tG(t^{\frac{1}{2}}))^{-\frac{1}{2}},
\end{equation}
for some constant $c >0$.

Fix a point $p_0 \in \Sigma$ and, for a fixed $j\in \enne$ define
$$
f_j(p)=\frac{u(p)-u(p_0)+1}{\varphi(\gamma(p))^{\frac{1}{j}}}.
$$
Then $f_j(p_0)=1/\varphi(\gamma(p_0))^{1/j}>0$. Moreover, since $u^* < +\infty$ and $\varphi(\gamma(p)) \ra +\infty$ as $p\ra\infty$, we have $\limsup_{p\ra\infty}f_j(p) \leq 0$.
Thus, $f_j$ attains a positive absolute maximum at $p_j \in \Sigma$. Iterating this procedure we produce a sequence
$\left\{p_j\right\}_{j\in\enne}\subset \Sigma$.
The proof of \eqref{omoriyau},(i) and \eqref{omoriyau},(ii) is the same as in \cite{pirise}, so we only prove \eqref{omoriyau},(iii). Proceeding as in Theorem 1.9 of \cite{pirise}, up to passing to a subsequence, we have $\lim_{j\ra+\infty}u(p_j)=u^*$.
If $\{p_j\}$ is contained in a compact set, then $p_j \ra \overline{p} \in \Sigma$ as $j\ra +\infty$ and $u$ attains its absolute maximum. Hence we have
$$
u(\overline{p})=u^*,\qquad \norm{\nabla u (\overline{p})}=0,\qquad \Hess{u}(\overline{p})\leq 0.
$$
In particular, since $P$ is positive semi-definite it holds that $Lu(\overline{p}) \leq 0$. Hence the sequence $p_j=\overline{p}$, for each $j$, satisfies all the requirements.
Consider now the case when $\{p_j\}$ diverges off a compact set, so that, according to \eqref{gamma1}, $\gamma(p_j) \ra +\infty$. Since $f_j$ attains a positive maximum at $p_j$ we have
$$
\mathrm{(i)}\  (\nabla \log f_j)(p_j)=0, \qquad \mathrm{(ii)}\   \Hess{\log f_j}(p_j) \leq 0.$$
A simple computation then gives
\begin{align*}
\Hess{u}(p_j)(v,v) \leq& \frac{1}{j}(u(p_j)-u(p_0)+1)\Big\{\frac{\varphi'(\gamma(p_j))}{\varphi(\gamma(p_j))}\Hess{\gamma}(p_j)(v,v)\\
&+\Big[\Big(\frac{1}{j}-1\Big)\Big(\frac{\varphi'(\gamma(p_j))}{\varphi(\gamma (p_j))}\Big)^2+\frac{\varphi''(\gamma(p_j))}{\varphi(\gamma(p_j))}\Big]\pair{\nabla \gamma(p_j),v}^2 \Big\}\\
\leq& \frac{1}{j}(u(p_j)-u(p_0)+1)\Big\{\frac{\varphi'(\gamma(p_j))}{\varphi(\gamma(p_j))}\Hess{\gamma}(p_j)(v,v)\\
&+\frac{1}{j}\Big(\frac{\varphi'(\gamma(p_j))}{\varphi(\gamma (p_j))}\Big)^2\pair{\nabla \gamma(p_j),v}^2 \Big\},
\end{align*}
for every $v\in T_{p_j}\Sigma$, where we have used \eqref{luis.1}. Let $\{e_1,\ldots,e_n\}\subset T_{p_j}\Sigma$ be an orthonormal basis of eigenvectors of $P(p_j)$ corresponding to
the eigenvalues $\mu_i(p_j)=\pair{P(p_j)e_i,e_i}\geq 0$. Then, for every $1\leq i\leq n$, we have
\begin{align*}
\pair{P\hess{u}(p_j)e_i,e_i}=& \mu_i(p_j)\Hess{u}(p_j)(e_i,e_i)\\
\leq &\frac{1}{j}(u(p_j)-u(p_0)+1)
\Big\{\frac{\varphi'(\gamma(p_j))}{\varphi(\gamma(p_j))}\pair{P\hess{\gamma}(p_j)e_i,e_i}\\
&+\frac{1}{j}\Big(\frac{\varphi'(\gamma(p_j))}{\varphi(\gamma (p_j))}\Big)^2\mu_i(p_j)\pair{\nabla\gamma(p_j),e_i}^2\Big\}.
\end{align*}
Taking traces here and using the fact that
\[
\pair{P\nabla\gamma,\nabla\gamma}=\sum_{i=1}^n\mu_i\pair{\nabla\gamma,e_i}^2\leq\mathrm{Tr}P\norm{\nabla\gamma}^2,
\]
we obtain that
\begin{align*}
Lu(p_j)\leq& \frac{1}{j}(u(p_j)-u(p_0)+1)\Big\{ \frac{\varphi'(\gamma(p_j))}{\varphi(\gamma(p_j))}L\gamma(p_j)\\
&+\frac{1}{j}\Big(\frac{\varphi'(\gamma(p_j))}{\varphi(\gamma (p_j))}\Big)^2\pair{P\nabla\gamma(p_j),\nabla\gamma(p_j)}\Big\}\\
\leq& \frac{1}{j}(u(p_j)-u(p_0)+1)\Big\{ \frac{\varphi'(\gamma(p_j))}{\varphi(\gamma(p_j))}L\gamma(p_j)\\
&+\frac{1}{j}\Big(\frac{\varphi'(\gamma(p_j))}{\varphi(\gamma (p_j))}\Big)^2\mathrm{Tr}P\norm{\nabla \gamma(p_j)}^2\Big\}.
\end{align*}
Since \eqref{gamma2} and \eqref{gamma3} hold outside a compact set, they hold at $p_j$ for $j$ sufficiently large. Then, using \eqref{luis.2},
\begin{align*}
Lu(p_j)\leq&\frac{1}{j}(u(p_j)-u(p_0)+1)\Big\{Bc+\frac{1}{j}c^2A^2CG(\gamma(p_j)^{\frac{1}{2}})^{-1}\Big\}\\
\leq & C \frac{u^*-u(p_0)+1}{j}
\end{align*}
for some constant $C>0$. Since the right hand side tends to zero as $j\rightarrow+\infty$, this proves condition (iii) in
\eqref{omoriyau}.
\end{proof}

Following the terminology introduced in \cite{pirise}, we introduce the next
\begin{definition}
\begin{rm}
Let $\Sigma$ be a Riemannian manifold and let $L$ be an operator as is Theorem \ref{maxprinc}. The Omori-Yau maximum principle is said to hold on $\Sigma$ for the operator $L$ if, for any function $u \in C^{2}(\Sigma)$ with
$u^*=\sup_{\Sigma}u < +\infty$, there exists a sequence $\left\{p_{j}\right\}_{j\in \enne} \subset \Sigma$ with the properties
\[
\mathrm{(i)}\   u(p_j)>u^*-\frac{1}{j},\ \mathrm{(ii)}\   \norm{\nabla u(p_j)}<\frac{1}{j}, \ \mathrm{(iii)}\   Lu(p_j)< \frac{1}{j}
\]
for every $j\in\enne$. Equivalently, for any function $u \in C^{2}(\Sigma)$ with
$u_*=\inf_{\Sigma}u > -\infty$, there exists a sequence $\left\{p_{j}\right\}_{j\in \enne} \subset \Sigma$ with the properties
\[
\mathrm{(i)}\   u(p_j)<u_*+\frac{1}{j},\ \mathrm{(ii)}\   \norm{\nabla u(p_j)}<\frac{1}{j}, \ \mathrm{(iii)}\   Lu(p_j)> -\frac{1}{j}
\]
for every $j\in\enne$.
\end{rm}
\end{definition}

The function theoretic approach to the generalized Omori-Yau maximum principle given in Theorem \ref{maxprinc} allows us
to apply it in different situations, where the choices of the functions $\gamma$ and $G$ are suggested by the geometric setting.
The next are two significant and useful examples of an intrinsic and extrinsic nature, respectively.

Let $(\Sigma,\pair{,})$ be a complete, non-compact Riemannian manifold and let $o\in\Sigma$ be a fixed reference point. Denote with $r(p)$ the
distance function from $o$ and set $\gamma(p)=r(p)^2$. Then $\gamma$ satisfies assumptions \eqref{gamma1} and \eqref{gamma2} of
Theorem \ref{maxprinc}. Furthermore, $\gamma$ is smooth within the cut locus of $o$.
Assume that the radial sectional curvature of $\Sigma$, that is, the sectional curvature of the 2-planes containing $\nabla r$, satisfies
\begin{equation}
\label{sectM}
K^{\mathrm{rad}}_{\Sigma}\geq -G(r),
\end{equation}
where $G$ be a smooth function on $[0,+\infty)$ even at the origin, i.e. $G^{(2k+1)}(0)=0$ for each $k=0,1,\ldots$, and satisfying
conditions (i)--(iv) listed in \eqref{condG}. Then assumption \eqref{gamma3} is satisfied.

Indeed, assuming that \eqref{sectM} holds, by the Hessian comparison theorem within the cut locus of $o$, one has
\begin{equation}
\label{luis.5}
\Hess{r}(p)(v,v)\leq\frac{\phi'(r(p))}{\phi(r(p))}(\norm{v}^2-\pair{\nabla r(p),v}^2)
\end{equation}
for every $v\in T_p\Sigma$, where $\phi(t)$ is the (positive) solution of the initial value problem
\begin{displaymath}
\left\{ \begin{array}{l}
\phi''-G\phi=0,\\
\phi(0)=0,\ \phi'(0)=1.
\end{array} \right.
\end{displaymath}
Now let
$$
\psi(t)=\frac{1}{\sqrt{G(0)}}\left(e^{\int_0^t\sqrt{G(s)}ds}-1\right).
$$
Then $\psi(0)=0$, $\psi'(0)=1$ and
\[
\psi''(t)-G(t)\psi(t)=\frac{1}{\sqrt{G(0)}}\left(G(t)+\frac{G'(t)}{2\sqrt{G(t)}}\,e^{\int_0^t\sqrt{G(s)}ds}\right)\geq 0.
\]
Hence, by the Sturm comparison theorem
\begin{equation}
\label{luis.6}
\frac{\phi'(t)}{\phi(t)}\leq \frac{\psi'(t)}{\psi(t)}=\sqrt{G(t)}\frac{e^{\int_0^t\sqrt{G(s)}ds}}{e^{\int_0^t\sqrt{G(s)}ds}-1}\leq
c\sqrt{G(t)}
\end{equation}
where the last inequality holds for a constant $c>0$ and $t$ sufficiently large. Therefore, if $r$ is sufficiently large
\[
\Hess{r}\leq c\sqrt{G(r)}(\pair{,}-dr\otimes dr).
\]
Since  $\Hess{\gamma}=2r\Hess{r}+2dr\otimes dr$, we obtain from here that
\begin{equation}
\label{luis.20}
\Hess{\gamma}\leq c\sqrt{\gamma G(\sqrt{\gamma})}\pair{,}
\end{equation}
for a constant $c$ and $\gamma$ sufficiently large.
Then, using the fact that $P$ is positive semi-definite
\[
L\gamma\leq c\mathrm{Tr}P\sqrt{\gamma G(\sqrt{\gamma})}
\]
for a constant $c$ and $\gamma$ sufficiently large. We have thus proved the following:

\begin{corollary}
\label{coroOY}
Let $(\Sigma,\pair{,})$ be a complete, non-compact Riemannian manifold whose radial sectional curvature satisfies condition \eqref{sectM}.
Then, the Omori-Yau maximum principle holds on $\Sigma$ for any semi-elliptic operator
$L=\mathrm{Tr}(P \circ \hess)$ with $\sup_\Sigma\mathrm{Tr}P<+\infty$.
\end{corollary}

On the other hand, the following example, of an extrinsic nature, will be useful in the sequel
for the case of \textit{properly immersed} hypersurfaces. Consider $\p^n$ a complete, non-compact, Riemannian manifold,
let $o\in\p^n$ be a reference point and denote by $\hat{r}$ the
distance function from $o$. We will assume that the radial sectional curvature of $\p^n$ satisfies the condition
\begin{equation}
\label{sectP}
K^{\mathrm{rad}}_{\mathbb{P}}\geq -G(\hat{r}),
\end{equation}
where $G$ is a smooth function on $[0,+\infty)$ even at the origin, i.e. $G^{(2k+1)}(0)=0$ for each $k=0,1,\ldots$, and satisfying
conditions (i)--(iv) listed in \eqref{condG}.
Let $f:\Sigma^n\ra M^{n+1}=I\times_{\rho} \p^n$ be a hypersurface. Observe that if $\Sigma$ is compact then every
immersion $f:\Sigma^n\ra I\times_{\rho} \p^n$  is proper and contained in a slab, and the
Omori-Yau maximum principle trivially holds on $\Sigma$ for any semi-elliptic operator. Assume then that $\Sigma$ is non-compact
and let $f:\Sigma^n\ra M^{n+1}=I\times_{\rho} \p^n$ be a properly immersed hypersurface which is contained in a slab, that is,
$f(\Sigma)\subset [t_1,t_2]\times\p^n$.

Let $\hat{\gamma}:\p^n\ra\erre$ be the function given by
$\hat{\gamma}(x)=\hat{r}(x)^2$ for every $x\in\p^n$, and set $\gamma:\Sigma\ra\erre$ for the associated function, defined as
\[
\gamma(p)=\tilde{\gamma}(f(p))=\hat{\gamma}(x(p))=\hat{r}(x(p))^2
\]
for every $p\in\Sigma$, where $\tilde{\gamma}(t,x)=\hat{\gamma}(x)$ and $f(p)=(h(p),x(p))$. Since $f$ is proper, if $p\ra\infty$ in $\Sigma$ then
$f(p)\ra\infty$ in $M^{n+1}=I\times_{\rho} \p^n$ , but being $f$ contained in a slab, this means that
$x(p)\ra\infty$ in $\p^n$. It follows that $\gamma(p)=\hat{r}(x(p))^2\ra+\infty$ as $p\ra\infty$ in $\Sigma$, and $\gamma$ satisfies
condition \eqref{gamma1} in Theorem \ref{maxprinc}.

Let us denote by $\tilde{\nabla}$, $\hat{\nabla}$ and $\nabla$ the Levi-Civita connection (and the gradient operators) in
$M^{n+1}$, $\p^n$ and $\Sigma^n$, respectively. Since $\gamma=\tilde{\gamma}\circ f$, along the immersion $f$ we have
\begin{equation}
\label{luis.14}
\tilde{\nabla}\tilde{\gamma}=\nabla\gamma+\pair{\tilde{\nabla}\tilde{\gamma},N}N
\end{equation}
where $N$ is a (local) smooth unit normal field along $f$.
On the other hand, from $\tilde{\gamma}(t,x)=\hat{\gamma}(x)$ we have
\[
\pair{\tilde{\nabla}\tilde{\gamma},T}=0,
\]
where, as usual, $T$ stands for the lift of $\partial_t$ to the product $I\times\p^n$, and
\[
\langle\tilde{\nabla}\tilde{\gamma},V\rangle=\pair{\hat{\nabla}\hat{\gamma},V}_{\p}
\]
for every $V$, where $V$ denotes the lift of a vector field $V\in TP$ to $I\times\p^n$. Since
\[
\langle\tilde{\nabla}\tilde{\gamma},V\rangle=\rho^2\pair{\tilde{\nabla}\tilde{\gamma},V}_{\p},
\]
we conclude from here that
\begin{equation}
\label{luis.13}
\tilde{\nabla}\tilde{\gamma}=\frac{1}{\rho^2}\hat{\nabla}\hat{\gamma}=\frac{2\hat{r}}{\rho^2}\hat{\nabla}\hat{r}.
\end{equation}
Therefore, since $\norm{\hat{\nabla}\hat{r}}=\rho\norm{\hat{\nabla}\hat{r}}_\p=\rho$ and $\rho(h)\geq\min_{[t_1,t_2]}\rho(t)>0$, along the immersion we have
\begin{equation}
\label{luis.17}
\norm{\nabla\gamma}\leq\norm{\tilde{\nabla}\tilde{\gamma}}=\frac{2\sqrt{\gamma}}{\rho(h)}\leq c\sqrt{\gamma}
\end{equation}
for a positive constant $c$. Thus, $\gamma$ also satisfies condition
\eqref{gamma2} in Theorem \ref{maxprinc}. In particular $\Sigma$ is complete (see \cite{pirise} pag 10).

Next, we will see that, under appropriate extrinsic restrictions, condition \eqref{gamma3} in Theorem \ref{maxprinc} is
also satisfied. From \eqref{luis.14} it follows that
\[
\Hess{\gamma}(X,X)={\Hess}\,{{\tilde{\gamma}}}(X,X)+\pair{\tilde{\nabla}\tilde{\gamma},N}\pair{AX,X}
\]
for every tangent vector field $X\in T\Sigma$. From \eqref{luis.13}
\begin{equation}
\label{luis.15}
\tilde{\nabla}_T\tilde{\nabla}\tilde{\gamma}=-\frac{\rho'}{\rho^3}\hat{\nabla}\hat{\gamma}=
-\mathcal{H}\tilde{\nabla}\tilde{\gamma},
\end{equation}
where $\mathcal{H}(t)=\rho'(t)/\rho(t)$. In particular, ${\Hess}\,{{\tilde{\gamma}}}(T,T)=0$. Then, writing $X=X^*+\pair{X,T}T$, where $X^*={\pi_{\p}}_*X$, we
have
\[
{\Hess}\,{{\tilde{\gamma}}}(X,X)={\Hess}\,{{\tilde{\gamma}}}(X^*,X^*)+2\pair{X,T}{\Hess}\,{{\tilde{\gamma}}}(X^*,T).
\]
From \eqref{luis.15} we have that
\[
{\Hess}\,{{\tilde{\gamma}}}(X^*,T)=-\mathcal{H}(h)\pair{\tilde{\nabla}\tilde{\gamma},X}=
-\mathcal{H}(h)\pair{\nabla\gamma,X}.
\]
On the other hand, using
\[
\tilde{\nabla}_{X^*}\tilde{\nabla}\tilde{\gamma}=
\frac{1}{\rho^2}\hat{\nabla}_{X^*}\hat{\nabla}\hat{\gamma}-\frac{\rho'}{\rho^3}\pair{\hat{\nabla}\hat{\gamma},X^*}T
\]
we also have
\[
{\Hess}\,{{\tilde{\gamma}}}(X^*,X^*)=\frac{1}{\rho^2}
\pair{\hat{\nabla}_{X^*}\hat{\nabla}\hat{\gamma},X^*}=
\pair{\hat{\nabla}_{X^*}\hat{\nabla}\hat{\gamma},X^*}_\p=
{\Hess}\,{{\hat{\gamma}}}(X^*,X^*).
\]
Summing up,
\begin{eqnarray}
\label{luis.16}
\Hess{\gamma}(X,X) & = & {\Hess}\,{{\hat{\gamma}}}(X^*,X^*)-2\mathcal{H}(h)\pair{\nabla\gamma,X}\pair{T,X}\\
\nonumber {} & {} & +\pair{\tilde{\nabla}\tilde{\gamma},N}\pair{AX,X}
\end{eqnarray}
for every tangent vector field $X\in T\Sigma$.

Observe that, using \eqref{luis.17},
\[
\left|\mathcal{H}(h)\pair{\nabla\gamma,X}\pair{T,X}\right|\leq\left|\mathcal{H}(h)\right|
\norm{\nabla\gamma}\norm{X}^2\leq c\sqrt{\gamma}\norm{X}^2.
\]
for a constant $c>0$, since $|\mathcal{H}(h)|\leq\max_{[t_1,t_2]}|\mathcal{H}(t)|$.
On the other hand, reasoning as we did before in deriving \eqref{luis.20}, from condition \eqref{sectP} and
using the Hessian comparison theorem for $\p^n$ it follows that, if $\gamma$ is sufficiently large, then
\[
{\Hess}\,{{\hat{\gamma}}}(X^*,X^*)\leq c\sqrt{{\gamma}G(\sqrt{{\gamma}})}\norm{X}^2
\]
for a certain positive constant $c$, where we are using the fact that
\[
\norm{X^*}_{\p}\leq\frac{1}{\inf_{\Sigma}\rho(h)}\norm{X}\leq\frac{1}{\min_{[t_1,t_2]}\rho(t)}\norm{X}.
\]
Therefore, since $\lim_{t\ra+\infty}G(t)=+\infty$, from \eqref{luis.16} we conclude
\begin{equation}
\label{luis.18}
\Hess{\gamma}(X,X)\leq c\sqrt{{\gamma}G(\sqrt{{\gamma}})}\norm{X}^2+\pair{\tilde{\nabla}\tilde{\gamma},N}\pair{AX,X}
\end{equation}
for every tangent vector field $X\in T\Sigma$, outside a compact subset of $\Sigma$.

Assume now that $\sup_\Sigma|H|<+\infty$. Tracing \eqref{luis.18} we obtain
\[
\Delta\gamma \leq nc\sqrt{{\gamma}G(\sqrt{{\gamma}})}+nH\pair{\tilde{\nabla}\tilde{\gamma},N}
\]
outside a compact set. Furthermore, by \eqref{luis.17}
\[
|H\pair{\tilde{\nabla}\tilde{\gamma},N}|\leq\sup_\Sigma|H|\norm{\tilde{\nabla}\tilde{\gamma}}\leq c\sqrt{\gamma}
\leq c\sqrt{{\gamma}G(\sqrt{{\gamma}})}
\]
for some constant $c>0$. Thus, we conclude that, outside a compact subset of $\Sigma$,
\[
\Delta\gamma \leq c \sqrt{{\gamma}G(\sqrt{{\gamma}})}
\]
for some constant $c>0$,
which means that condition \eqref{gamma3} in Theorem \ref{maxprinc} is fulfilled for the Laplacian operator. Therefore,
the Omori-Yau maximum principle holds on $\Sigma$ for the Laplacian.

On the other hand, if we assume instead that $\sup_\Sigma\|A\|^2<+\infty$ then, using again \eqref{luis.17}, we have
\[
|\pair{\tilde{\nabla}\tilde{\gamma},N}\pair{AX,X}|\leq\norm{\tilde{\nabla}\tilde{\gamma}}\norm{A}\norm{X}^2
\leq c\sqrt{{\gamma}G(\sqrt{{\gamma}})}\norm{X}^2
\]
for a positive constant $c$, if $\gamma$ is sufficiently large. From \eqref{luis.18} we therefore obtain
\begin{equation}
\label{luis.19}
\Hess{\gamma}(X,X)\leq c\sqrt{{\gamma}G(\sqrt{{\gamma}})}\norm{X}^2,
\end{equation}
for every tangent vector field $X\in T\Sigma$, outside a compact subset of $\Sigma$. Thus, if $P$ is a positive
semi-definite operator with $\sup_\Sigma\mathrm{Tr}P<+\infty$, we conclude from here that
\[
L\gamma\leq nc\sup_\Sigma\mathrm{Tr}P\sqrt{{\gamma}G(\sqrt{{\gamma}})}
\]
if $\gamma$ is sufficiently large, which means that condition \eqref{gamma3} in Theorem \ref{maxprinc} is fulfilled for
the operator $L$. Therefore the Omori-Yau maximum principle holds on $\Sigma$ for $L$. We summarize the above discussion in the following:
\begin{corollary}
\label{OYproperly}
Let $\p^n$ be a complete, non-compact, Riemannian manifold whose radial sectional curvature satisfies the condition \eqref{sectP}.
Let $f:\Sigma^n\ra M^{n+1}=I\times_{\rho} \p^n$ be a properly immersed hypersurface contained in a slab.
\begin{enumerate}
\item  If $\sup_\Sigma|H|<+\infty$, then $\Sigma$ is complete and the Omori-Yau maximum principle holds on $\Sigma$ for the Laplacian.
\item If $\sup_\Sigma\|A\|<+\infty$, then $\Sigma$ is complete and the Omori-Yau maximum principle holds on $\Sigma$ for any semi-elliptic operator
$L=\mathrm{Tr}(P \circ \hess)$ with $\sup_\Sigma\mathrm{Tr}P<+\infty$.
\end{enumerate}
\end{corollary}
\begin{remark}\label{remarkmarco}
\rm{
From the equality
$$
\norm{A}^2=n^2H_1^2-n(n-1)H_2
$$
it follows that under the assumption $\inf_\Sigma H_2>-\infty$ the condition $\sup_\Sigma \norm{A}^2<+\infty$ is
equivalent to $\sup_\Sigma|H_1|<+\infty$.
}
\end{remark}

\section{Curvature estimates for hypersurfaces in warped products}
In this section we will derive some estimates for the $k$-mean curvatures of a hypersurface in a slab of a warped product space. Towards this aim we need the next computational
\begin{prop}
\label{propsigma}
Let $f: \Sigma^n \ra M^{n+1}=I \times_{\rho} \p^n$ be an isometric immersion into a warped product space. Let $h$ be the height function and define
$$
\sigma(t)=\int_{t_0}^t \rho(u) \di u.
$$
Then
\begin{equation}\label{lrh}
L_k h=\mathcal{H}(h)(c_kH_k-\pair{P_k\nabla h,\nabla h})+c_k\Theta H_{k+1},
\end{equation}
and
\begin{equation}\label{lrsigma}
L_k \sigma(h)=c_k\rho(h)(\mathcal{H}(h)H_k+\Theta H_{k+1}),
\end{equation}
where $c_k=(n-k){n \choose k}=(k+1){n \choose {k+1}}$, $\mathcal{H}(t)=\rho'(t)/\rho(t)$, and $\Theta=\pair{N,T}$ is the angle function.
\end{prop}
\begin{proof}
The gradient of $\pi_{I} \in C^{\infty}(M)$ is $\overline{\nabla}\pi_{I}=T$, hence:
\begin{equation}
\label{luis.11}
\nabla h=(\overline{\nabla}\pi_{I})^\top=T-\Theta N.
\end{equation}
Recall that the Levi-Civita connection of a warped product satisfies
$$
\overline{\nabla}_X T=\mathcal{H}(X-\pair{X,T}T), \qquad \text{for any } X \in TM.
$$
Thus
$$\overline{\nabla}_X\nabla h=\mathcal{H}(h)(X-\pair{X,T}T)-X(\Theta)N+\Theta AX,
$$
for any $X \in T\Sigma$. Then
\begin{equation}\label{hessh}
\hess{h}(X)=\nabla_X\nabla h=\mathcal{H}(h)(X-\pair{X,\nabla h}\nabla h)+\Theta AX,
\end{equation}
where $\nabla$ denotes the Levi-Civita connection on $\Sigma^n$. Let $\left\{e_1,...,e_n\right\}$ be a local orthonormal frame on
$\Sigma$. Then, using the expressions of the traces of $P_k$ and $P_kA$
\begin{align*}
L_k h = &\mathrm{Tr}(P_k \circ \hess{h})=\sum_i \pair{P_k \hess{h}(e_i),e_i}\\
= & \mathcal{H}(h)\big(\mathrm{Tr}P_k-\pair{P_k\nabla h,\nabla h}\big)+\Theta \mathrm{Tr}(P_k A)\\
= & \mathcal{H}(h)(c_kH_k-\pair{P_k\nabla h,\nabla h})+c_k\Theta H_{k+1}.
\end{align*}
On the other hand, since $\nabla \sigma(h)=\rho(h)\nabla h$, we have
$$
\hess{\sigma(h)}(X)=\rho'(h)\pair{\nabla h,X}\nabla h+\rho(h)\hess{h}(X).
$$
Therefore
\begin{align*}
L_k \sigma(h)=&\mathrm{Tr}(P_k \circ \hess{\sigma(h)})=\sum_i \pair{P_k \hess{\sigma(h)}(e_i),e_i}\\
=&\rho'(h)\pair{P_k \nabla h,\nabla h}+\rho(h)\mathcal{H}(h)\big(\mathrm{Tr}P_k-\pair{P_k\nabla h,\nabla h}\big)\\
&+\rho(h)\Theta\mathrm{Tr}(P_k A)\\
= & c_k\rho(h)(\mathcal{H}(h)H_k+\Theta H_{k+1}).
\end{align*}
\end{proof}

As a first application of the computations above, we derive the following:
\begin{theorem}
\label{thmC}
Let $f: \Sigma^n \ra I \times_{\rho} \p^n$ be an immersed hypersurface. If the Omori-Yau maximum principle holds on
$\Sigma$ for the Laplacian and $h^*=\sup_\Sigma h<+\infty$, then
\[
\sup_\Sigma|H|\geq\inf_\Sigma\mathcal{H}(h).
\]
\end{theorem}
In particular, and as an application of Corollary \ref{OYproperly}, we deduce the following result, which generalizes
Theorem 2 in \cite{aliasdajczer2}.
\begin{corollary}
\label{corothmC}
Let $\p^n$ be a complete, non-compact, Riemannian manifold whose radial sectional curvature satisfies condition
\eqref{sectP}. If $f:\Sigma^n\ra M^{n+1}=I\times_{\rho} \p^n$ is a properly immersed hypersurface contained in
a slab, then
\begin{equation}
\label{luis.22}
\sup_\Sigma|H|\geq \inf_\Sigma\mathcal{H}(h).
\end{equation}
In other words, there is no properly immersed hypersurface contained in a slab $[t_1,t_2]\times\p^n$ with
\[
\sup_\Sigma|H|<\inf_{[t_1,t_2]}\mathcal{H}(t).
\]
\end{corollary}
For the proof of Corollary \ref{corothmC}, observe that if $\sup_\Sigma|H|=+\infty$ then the inequality \eqref{luis.22} trivially holds.
On the other hand, if $\sup_\Sigma|H|<+\infty$ then by Corollary \ref{OYproperly} we know that the Omori-Yau maximum principle holds
on $\Sigma$ and the result follows from Theorem \ref{thmC}.

\begin{proof}[Proof of Theorem \ref{thmC}]
Since $h$ is bounded from above, we may find a sequence $\{q_j\} \subset \Sigma^n$ such that
\begin{eqnarray*}
\lim_{j\ra +\infty}h(q_j) & = & h^*:=\sup h ,\\
\norm{\nabla h(q_j)}^2& = & 1-\Theta^2(q_j) < \Big(\frac{1}{j}\Big)^2,\\
\Delta h(q_j) & = & \mathcal{H}(h(q_j))(n-\norm{\nabla h(q_j)}^2)+nH(q_j)\Theta(q_j)<\frac{1}{j}.
\end{eqnarray*}
Then
$$
\frac{1}{j}>\Delta h(q_j)\geq \mathcal{H}(h(q_j))(n-\norm{\nabla h(q_j)}^2)-n\sup_{\Sigma}|H|.
$$
Making $j\ra+\infty$ we get
$$
0\geq\mathcal{H}(h^*)-\sup_\Sigma |H|,
$$
so that
$$
\sup_\Sigma |H|\geq\mathcal{H}(h^*)\geq \inf_\Sigma\mathcal{H}(h).
$$
\end{proof}

\begin{corollary}
Let $\p^n$ be a complete, non-compact, Riemannian manifold whose radial sectional curvature satisfies condition
\eqref{sectP}. If $f:\Sigma^n\ra M^{n+1}=I\times_{\mathrm{e}^t} \p^n$ is a parabolic, properly immersed hypersurface with constant mean curvature $|H|\leq 1$ contained in
a slab, then $f(\Sigma)$ is slice.
\end{corollary}
For the proof of this corollary observe that from Corollary \ref{corothmC} it must be $|H|=1$. Choose the orientation so that $H=1$. In this case $\sigma(h)=\mathrm{e}^h$ and by \eqref{lrsigma}
$$
\Delta \mathrm{e}^h=n\mathrm{e}^h(1+\Theta)\geq0.
$$
Therefore, since $\mathrm{e}^h\leq\mathrm{e}^{h^*}$ it follows that $\mathrm{e}^h$ is a subharmonic function on $\Sigma$ which is bounded from above. The conclusion now follows from parabolicity.

For the next results, we normalize the operators $P_k$ to the following
$$
\hat{P}_k=\frac{1}{H_k} P_k
$$
where, of course, we are assuming $H_k>0$. We will denote by $\hat{L}_k$ the corresponding differential operator, that is
$$
\hat{L}_k=\mathrm{Tr}(\hat{P}_k\circ\hess).
$$
Observe that
$$
\mathrm{Tr}(\hat{P_k})=c_k
$$
so that the operators $\hat{P}_k$ have trace always bounded from above. With this preparation we state the next
\begin{theorem}
Let $f: \Sigma^n \ra I \times_{\rho} \p^n$ be an immersed hypersurface with $H_2>0$. If the Omori-Yau
maximum principle holds on $\Sigma$ for $\hat{L}_1$ and $h^*=\sup_\Sigma h<+\infty$, then
\[
\sup_\Sigma H^{1/2}_2\geq\inf_\Sigma\mathcal{H}(h).
\]
\end{theorem}
\begin{proof}
We may assume without loss of generality that $\sup_\Sigma H_2<+\infty$ and $\inf_\Sigma\mathcal{H}(h)\geq 0$. Otherwise the desired conclusion trivially holds.
Using the basic inequality $H_1\geq \sqrt{H_2}$ we know that $H_1>0$ and $\hat{L}_1$ is a well defined elliptic operator.
Since $h$ is bounded from above and $\sup_\Sigma\sigma(h)=\sigma(h^*)$, we may find a sequence $\{q_j\} \subset \Sigma^n$ such that
\begin{eqnarray*}
\lim_{j\ra +\infty}(\sigma\circ h)(q_j) & = & \sigma(h^*):=\sup (\sigma\circ h) ,\\
\norm{\nabla (\sigma\circ h)(q_j)}^2& =&\rho(h(q_j))^2(1-\Theta^2(q_j)) <\Big(\frac{1}{j}\Big)^2,\\
\hat{L}_1(\sigma\circ h)(q_j)&<&\frac{1}{j}.
\end{eqnarray*}
Then we have
\begin{align*}
\frac 1j >&\hat{L}_1 (\sigma\circ h)(q_j)= n(n-1)\rho(h(q_j))\left(\mathcal{H}(h(q_j))+\Theta(q_j)\frac{H_2}{H_1}(q_j)\right)\\
\geq& n(n-1)\rho(h(q_j))\left(\mathcal{H}(h(q_j))-\frac{H_2}{H_1}(q_j)\right)\\
\geq& n(n-1)\rho(h(q_j))\left(\mathcal{H}(h(q_j))-\sqrt{H_2}(q_j)\right)
\end{align*}
Observe that
$$\lim_{j\ra +\infty}(\sigma\circ h)(q_j) = \sigma(h^*):=\sup (\sigma\circ h)$$
implies $\lim_{j\ra+\infty}h(q_j)=h^*$, because $\sigma(t)$ is strictly increasing.
Making $j\ra+\infty$, and, if necessary, up to passing to a subsequence we get
$$
0\geq \mathcal{H}(h^*)-\sup_\Sigma\sqrt{H_2}.
$$
So
$$
\sup_\Sigma H_2^{1/2}\geq\mathcal{H}(h^*)\geq \inf_\Sigma\mathcal{H}(h).
$$
\end{proof}
As a consequence of the previous theorem and of Corollary \ref{OYproperly} (see also Remark \ref{remarkmarco}) we have
\begin{corollary}
Let $\p^n$ be a complete, non-compact, Riemannian manifold whose radial sectional curvature satisfies condition
\eqref{sectP}. If $f:\Sigma^n\ra M^{n+1}=I\times_{\rho} \p^n$ is a properly immersed hypersurface with $H_2>0$,
$\sup_\Sigma|H_1|<+\infty$ and contained in a slab, then
\[
\sup_\Sigma H^{1/2}_2\geq \inf_\Sigma\mathcal{H}(h).
\]
In other words, there is no properly immersed hypersurface with $H_2>0$ and $\sup_\Sigma |H_1|<+\infty$ contained in a slab
$[t_1,t_2]\times\p^n$ with
\[
\sup_\Sigma H^{1/2}_2<\inf_{[t_1,t_2]}\mathcal{H}(t).
\]
\end{corollary}
In the next theorem, the existence of an elliptic point enable us to guarantee that $H_{k-1}$ is strictly positive and to guarantee ellipticity of the operator $\hat{L}_{k-1}$. Reasoning as in the previous results gives the next
\begin{theorem}
Let $f: \Sigma^n \ra I \times_{\rho} \p^n$ be an immersed hypersurface having an elliptic point, with $H_k>0$. If the Omori-Yau
maximum principle holds on $\Sigma$ for $\hat{L}_{k-1}$, with $3\leq k\leq n$, and $h^*<+\infty$ then
\[
\sup_\Sigma H^{1/k}_{k}\geq\inf_\Sigma\mathcal{H}(h).
\]
\end{theorem}
\begin{corollary}
Let $\p^n$ be a complete, non-compact, Riemannian manifold whose radial sectional curvature satisfies condition
\eqref{sectP}. Assume that $f:\Sigma^n\ra M^{n+1}=I\times_{\rho} \p^n$ is a properly immersed hypersurface having an elliptic
point, with $H_k>0$ and $\sup_\Sigma |H_1|<+\infty$. If $f(\Sigma)$ is contained in a slab, then
\[
\sup_\Sigma H^{1/k}_k\geq \inf_\Sigma\mathcal{H}(h)
\]
for every $3\leq k\leq n$. In other words, there is no properly immersed hypersurface having an elliptic point, with $H_k>0$, $\sup_\Sigma |H_1|<+\infty$ and
contained in a slab $[t_1,t_2]\times\p^n$ with
\[
\sup_\Sigma H^{1/k}_k<\inf_{[t_1,t_2]}\mathcal{H}(t).
\]
\end{corollary}

\section{Hypersurfaces with constant 2-mean curvature}

In this section we will derive some applications for hypersurfaces with positive constant 2-mean curvature $H_2$.
Before stating the main results, let us introduce an auxiliary lemma that will be useful in the sequel.
\begin{lemma}\label{lemmamainthm}
Let $f:\Sigma^n \ra I\times_{\rho} \p^n$ be a hypersurface with non-vanishing mean curvature which is contained in a slab.
Assume that $\mathcal{H}'\geq 0$ and that the angle function $\Theta$ does not change sign. Choose on $\Sigma$ the orientation so
that $H_1>0$. Suppose the Omori-Yau maximum principle for the Laplacian holds on $\Sigma$. We have that
\begin{align*}
\mathrm{(i)} & \text{ if } \Theta\leq 0,\qquad \text{then } \mathcal{H}(h) \geq 0,\\
\mathrm{(ii)} &\text{ if } \Theta\geq 0,\qquad \text{then }\mathcal{H}(h) \leq 0.
\end{align*}
\end{lemma}
\begin{proof}
Since $h$ is bounded from below and the Omori-Yau maximum principle for the Laplacian operator holds on $\Sigma$, we can find a sequence $\{p_j\} \subset \Sigma^n$ such that
\begin{eqnarray*}
\lim_{j\ra +\infty}h(p_j) & = & h_*:=\inf h ,\\
\norm{\nabla h(p_j)}^2& = & 1-\Theta^2(p_j) < \Big(\frac{1}{j}\Big)^2,\\
\Delta h(p_j) & = & \mathcal{H}(h(p_j))(n-\norm{\nabla h(p_j)}^2)+nH_1(p_j)\Theta(p_j)>-\frac{1}{j}.
\end{eqnarray*}
Then
\begin{equation}
\label{luis.3}
-nH_1(p_j)\Theta(p_j)<\frac{1}{j}+\mathcal{H}(h(p_j))(n-\norm{\nabla h(p_j)}^2).
\end{equation}
Similarly, since $h$ is bounded from above, we can also find a second sequence $\{q_j\} \subset \Sigma^n$ such that
\begin{eqnarray*}
\lim_{j\ra +\infty}h(q_j) & = & h^*:=\sup h ,\\
\norm{\nabla h(q_j)}^2& = & 1-\Theta^2(q_j) < \Big(\frac{1}{j}\Big)^2,\\
\Delta h(q_j) & = & \mathcal{H}(h(q_j))(n-\norm{\nabla h(q_j)}^2)+nH_1(q_j)\Theta(q_j)<\frac{1}{j}.
\end{eqnarray*}
Then
\begin{equation}
\label{luis.4}
-nH_1(q_j)\Theta(q_j)>-\frac{1}{j}+\mathcal{H}(h(q_j))(n-\norm{\nabla h(q_j)}^2).
\end{equation}

Assume first that $\Theta\leq 0$. Since $\lim_{j\ra +\infty} -\Theta (p_j)=-\sgn\Theta=1>0$, we have $-\Theta(p_j)>0$ for sufficiently
large $j$. Since $H_1(p_j)>0$, using \eqref{luis.3} it follows from \eqref{luis.3} that
\[
0\leq\liminf_{j\ra+\infty}\Big(-H_1(p_j)\Theta(p_j)\Big)\leq\mathcal{H}(h_*).
\]
Therefore $\mathcal{H}(h_*)\geq 0$ and, by $\mathcal{H}'\geq 0$, we conclude that
$$
\mathcal{H}(h)\geq\mathcal{H}(h_*)\geq 0.
$$
Assume now that $\Theta\geq 0$ then $\lim_{j\ra+\infty}\Theta(q_j)=\sgn\Theta=1>0$, so that
$\Theta(q_j)>0$ for sufficiently large $j$. Therefore, since $H_1(q_j)>0$, from \eqref{luis.4} we deduce
\[
0\leq\liminf_{j\ra+\infty}\Big(H_1(q_j)\Theta(q_j)\Big)\leq-\mathcal{H}(h^*).
\]
Therefore $\mathcal{H}(h^*)\leq 0$ and, by $\mathcal{H}'\geq 0$, we conclude that
$$
\mathcal{H}(h)\leq\mathcal{H}(h^*)\leq 0.
$$
This concludes the proof.
\end{proof}

In the rest of this section we will work basically with the operator $L_1$. We will assume that $H_2$ is a positive constant. Recall that this implies, in this case, that the immersion is two-sided. We can choose the normal unit vector $N$ on $\Sigma$ such that $H_1>0$ and the operator $L_1$ is
elliptic (see the discussion in the Preliminaries).

Let $\sigma(t)=\int_{t_0}^t \rho(s) \di s$. By Proposition \ref{propsigma} we know that
\begin{eqnarray}
\label{luis.12}
\nonumber \Delta \sigma(h) & = & n \rho(h)(\mathcal{H}(h)+\Theta H_1),\\
L_1 \sigma(h) & = & n(n-1)\rho(h)(\mathcal{H}(h)H_1+\Theta H_2).
\end{eqnarray}
Therefore,
\begin{equation}
\label{luis.10}
\mathcal{L}_1\sigma(h)=n(n-1)\rho(h)(\mathcal{H}(h)^2-\Theta^2H_2),
\end{equation}
where $\mathcal{L}_1$ is the operator given by
$$
\mathcal{L}_1=(n-1)\mathcal{H}(h) \Delta - \Theta L_1=\mathrm{Tr}(\mathcal{P}_1\circ \hess),
$$
with
\[
\mathcal{P}_1=(n-1)\mathcal{H}(h)I-\Theta P_1.
\]

Let us now state the first main result of this section, which extends Theorem 2.4 in \cite{aliasdajczer}
to the case of constant 2-mean curvature $H_2$.
\begin{theorem}
\label{mainthmcompact}
Let $f: \Sigma^n \ra I \times_{\rho} \p^n$ be a compact hypersurface of constant positive $2$-mean curvature $H_2$.
If $\mathcal{H}'(t)\geq 0$ and the angle function $\Theta$ does not change sign, then $\p^n$ is necessarily
compact and $f(\Sigma^n)$ is a slice.
\end{theorem}
\begin{proof}
As indicated above, we choose the orientation of $\Sigma$ so that $H_1>0$. Since $\Sigma^n$ is compact, we may apply Lemma \ref{lemmamainthm}.
Let us consider first the case where $\Theta\leq 0$, for which $\mathcal{H}(h)\geq 0$. Thus, the
operator $\mathcal{P}_1$ is positive semi-definite or, equivalently,  $\mathcal{L}_1$ is semi-elliptic.

Since $\Sigma$ is compact, there exist points
$\pmax\in\Sigma$ and $\pmin\in\Sigma$ such that
$$h(\pmax)=h^*=\max_\Sigma h \quad \mathrm{and} \quad h(\pmin)=h_*=\min_\Sigma h.$$
Therefore, $\|\nabla h(\pmax)\|=\|\nabla h(\pmin)\|=0$, which yields $$\Theta(\pmax)=\Theta(\pmin)=-1$$ because of
\eqref{luis.11}. Observe that
$$(\sigma\circ h)^*=\max_\Sigma (\sigma\circ h)=\sigma(h^*)=\sigma(h(\pmax))$$ and
$$(\sigma\circ h)_*=\min_\Sigma (\sigma\circ h)=\sigma(h_*)=\sigma(h(\pmin)),$$ because $\sigma(t)$ is strictly
increasing. In particular,
\[
\Hess{\sigma(h)}(\pmax)\leq 0 \quad \mathrm{and} \quad \Hess{\sigma(h)}(\pmin)\geq 0.
\]
Taking into account that $\mathcal{P}_1$ is positive semi-definite, yields
\[
\mathcal{L}_1\sigma(h)(\pmax)=n(n-1)\rho(h^*)(\mathcal{H}(h^*)^2-H_2)\leq 0
\]
and
\[
\mathcal{L}_1\sigma(h)(\pmin)=n(n-1)\rho(h_*)(\mathcal{H}(h_*)^2-H_2)\geq 0.
\]
Then, since $\mathcal{H}(h)\geq 0$ on $\Sigma$, we obtain
$$\mathcal{H}(h_*)\geq H_2^{1/2}\geq \mathcal{H}(h^*).$$ On the other hand, by $\mathcal{H}'\geq 0$ we also have
$\mathcal{H}(h_*)\leq\mathcal{H}(h^*)$. Thus the validity of the equality $\mathcal{H}(h_*)=\mathcal{H}(h^*)$ and
$\mathcal{H}(h)=H_2^{1/2}$ is constant on $\Sigma$. By \eqref{luis.12}, using the basic inequality
$H_1\geq H_2^{1/2}$ and the fact that $\Theta\geq -1$, we obtain
\begin{eqnarray*}
L_1 \sigma(h) & = & n(n-1)\rho(h)H_2^{1/2}(H_1+\Theta H^{1/2}_2)\\
{} & \geq & n(n-1)\rho(h)H_2^{1/2}(H_1-H^{1/2}_2)\geq 0.
\end{eqnarray*}
That is, $L_1\sigma(h)\geq 0$ on the compact manifold $\Sigma$. Thus, by the maximum principle applied to the
elliptic operator $L_1$ we conclude that $\sigma(h)$, and hence $h$, is constant.

Finally, in the case where $\Theta\geq 0$ we know from Lemma \ref{lemmamainthm} that $\mathcal{H}(h)\leq 0$ on
$\Sigma$, so that the operator $-\mathcal{L}_1$ is semi-elliptic. The proof then follows as in the case $\Theta\leq 0$,
working with $-\mathcal{L}_1$ instead of $\mathcal{L}_1$.
\end{proof}

In our next result, we consider the case of complete (and non-compact) hypersurfaces, extending Theorem 2.9 in
\cite{aliasdajczer} to the case of constant 2-mean curvature hypersurfaces.
\begin{theorem}\label{mainthm}
Let $f: \Sigma^n \ra I \times_{\rho} \p^n$ be a complete hypersurface of constant positive $2$-mean curvature $H_2$ such that
\begin{equation}
\label{sectional}
K^{\mathrm{rad}}_{\Sigma}\geq-G(r).
\end{equation}
Here $G$ is a smooth function on $[0,+\infty)$ which is even at the origin and satisfies conditions (i)--(iv) listed in
Theorem \ref{maxprinc}. Assume that $\sup_{\Sigma}|H_1|<+\infty$ and that $\Sigma$ is contained in a slab, that is,
$$
f(\Sigma^n) \subset [t_1,t_2] \times \p^n,
$$
where $t_1,t_2 \in I$ are finite. If $\mathcal{H}'(t)>0$ almost everywhere and the angle function $\Theta$ does not change sign,
then $f(\Sigma^n)$ is a slice.
\end{theorem}

\begin{proof}
Choose the orientation of $\Sigma$ so that $H_1>0$. By Corollary \ref{coroOY} we know that the Omori-Yau maximum principle holds on
$\Sigma$ for the Laplacian operator, so that we may apply Lemma \ref{lemmamainthm}.

In the case where $\Theta\leq 0$, by Lemma
\ref{lemmamainthm} we have $\mathcal{H}(h)\geq 0$, and therefore the operator $\mathcal{P}_1$ is positive semi-definite. In other
words, the differential operator $\mathcal{L}_1$ is semi-elliptic.
Furthermore
\[
\mathrm{Tr}\mathcal{P}_1=n(n-1)\mathcal{H}(h)-n(n-1)H_1\Theta \leq n(n-1)(\mathcal{H}(h^*)+H_1^*),
\]
where $h^*=\sup_\Sigma h<+\infty$ and $H_1^*=\sup_{\Sigma}H_1<+\infty$. Hence by Corollary \ref{coroOY} we know that the Omori-Yau maximum principle holds
on $\Sigma$ for the operator $\mathcal{L}_1$.

Since $\sup_\Sigma\sigma(h)=\sigma(h^*)<+\infty$, there exists a sequence $\{p_j\}_{j\in\enne}\subset\Sigma$ such that
\begin{align*}
(i) \quad & \lim_{j\ra +\infty} \sigma(h(p_j))=\sup_\Sigma\sigma(h)=\sigma(h^*),\\
(ii) \quad & \norm{\nabla(\sigma \circ h)(p_j)}=\rho(h(p_j))\norm{\nabla h(p_j)}<\frac{1}{j},\\
(iii) \quad & \mathcal{L}_1(\sigma \circ h)(p_j)<\frac 1j.
\end{align*}
Observe that condition (i) implies that $\lim_{j\ra +\infty}h(p_j)=h^*$, because $\sigma(t)$ is strictly increasing. Thus by
condition (ii) we also have $\lim_{j\ra +\infty}\norm{\nabla h(p_j)}=0$.
Therefore
$$
\mathcal{L}_1 \sigma(h)(p_j)=n(n-1)\rho(h(p_j))(\mathcal{H}(h(p_j))^2-\Theta^2(p_j)H_2)<\frac{1}{j},
$$
and taking the limit for $j\ra +\infty$ and observing that $\Theta^2(p_j)=1-\norm{\nabla h(p_j)}^2 \ra 1$ as $j\ra +\infty$, we find
$$
\mathcal{H}(h^*)^2-H_2\leq 0.
$$

On the other hand, since $h$ is also bounded from below, $\inf_\Sigma\sigma(h)=\sigma(h_*)>-\infty$, where
$h_*=\inf_\Sigma h>-\infty$. Thus, we can find a sequence $\{q_j\}_{j\in\enne} \subset \Sigma$ such that
\begin{align*}
(i) \quad & \lim_{j\ra +\infty} \sigma(h(q_j))=\inf_\Sigma\sigma(h)=\sigma(h_*),\\
(ii) \quad & \norm{\nabla(\sigma \circ h)(q_j)}=\rho(h(q_j))\norm{\nabla h(q_j)}<\frac{1}{j},\\
(iii) \quad & \mathcal{L}_1(\sigma \circ h)(q_j)>-\frac 1j.
\end{align*}
Hence, proceeding as above and using that
$$
\mathcal{L}_1 \sigma(h)(q_j)=n(n-1)\rho(h(q_j))(\mathcal{H}(h(q_j))^2-\Theta^2(q_j)H_2)>-\frac{1}{j},
$$
we find
$$
\mathcal{H}(h_*)^2-H_2\geq 0.
$$
Thus $\mathcal{H}(h_*)^2\geq \mathcal{H}(h^*)^2$ and, taking into account that $\mathcal{H}(h_*),\mathcal{H}(h^*)\geq 0$, this gives
$\mathcal{H}(h_*)\geq \mathcal{H}(h^*)$. Therefore, since $\mathcal{H}(t)$ is an increasing function we conclude that $h^*=h_*$.

Finally, let us consider the case where $\Theta\geq 0$. By Lemma \ref{lemmamainthm} we find that $\mathcal{H}(h)\leq 0$ and then
the operator $-\mathcal{L}_1$ is semi-elliptic. Moreover
$$
\mathrm{Tr}(-\mathcal{P}_1)=-n(n-1)\mathcal{H}(h)+n(n-1)H_1\Theta \leq n(n-1)(-\mathcal{H}(h_*)+H_1^*).
$$
Hence the trace of $-\mathcal{P}_1$ is bounded from above and by Corollary \ref{coroOY} the Omori-Yau maximum principle holds for
the operator $-\mathcal{L}_1$. Proceeding as above we arrive at the two inequalities
\[
H_2-\mathcal{H}(h_*)^2\geq 0 \quad \text{ and } \quad H_2-\mathcal{H}(h^*)^2\leq 0.
\]
Thus $\mathcal{H}(h_*)^2 \leq \mathcal{H}(h^*)^2$. Since $\mathcal{H}(h_*), \mathcal{H}(h^*)\leq 0$, this implies
$\mathcal{H}(h_*)\geq \mathcal{H}(h^*)$. But $\mathcal{H}(t)$ being increasing, this gives $h_*=h^*$ concluding the proof.
\end{proof}

In particular, Theorem \ref{mainthm} remains true if we replace condition \eqref{sectional} by the stronger condition
of $\Sigma^n$ having radial sectional curvature bounded from below by a constant. This happens, for instance, when the sectional
curvature of $\p^n$ is itself bounded from below. This observation yields the next
\begin{corollary}
\label{coromainthm}
Let $\p^n$ be a complete Riemannian manifold with sectional curvature bounded from below and let
$f:\Sigma^n \ra I \times_{\rho} \p^n$ be a complete hypersurface of constant positive $2$-mean curvature $H_2$.
Assume that $\sup_{\Sigma}|H_1|<+\infty$ and that $\Sigma$ is contained in a slab, that is,
$$
f(\Sigma^n) \subset [t_1,t_2] \times \p^n,
$$
where $t_1,t_2 \in I$ are finite. If $\mathcal{H}'(t)>0$ almost everywhere and the angle function $\Theta$ does not change sign,
then $f(\Sigma^n)$ is a slice.
\end{corollary}

As already observed, for the proof of Corollary \ref{coromainthm}, it suffices to show that $K^{\mathrm{rad}}_\Sigma$ is bounded from below
by a constant, and the result then follows from Theorem \ref{mainthm}. Actually, we can prove the following stronger
result, which will be useful in the sequel.
\begin{lemma}
\label{lemmacurvature}
Let $\p^n$ be a Riemannian manifold with sectional curvature bounded from below and let
$f:\Sigma^n \ra I \times_{\rho} \p^n$ be an immersed hypersurface.
Assume that $\sup_{\Sigma}\|A\|^2<+\infty$ and that $\Sigma$ is contained in a slab. Then the sectional curvature of
$\Sigma$ is bounded from below by a constant.
\end{lemma}
Given the validity of Lemma \ref{lemmacurvature}, taking into account the equality
\[
\|A\|^2=\text{Tr}(A^2)=n^2H_1^2-n(n-1)H_2
\]
it follows that, under the assumptions of Corollary \ref{coromainthm},
$$\sup_{\Sigma}\|A\|^2\leq n^2(\sup_{\Sigma}H_1)^2-n(n-1)H_2<+\infty.$$
Thus $K^{\mathrm{rad}}_\Sigma$ is bounded from
below by a constant.
\begin{proof}[Proof of Lemma \ref{lemmacurvature}]
Recall that the Gauss equation for a hypersurface $f:\Sigma^n \rightarrow M^{n+1}$ is given by
$$
\pair{\R(X,Y)Z,V}=\pair{\overline{\R}(X,Y)Z,V}-\pair{AY,V}\pair{AX,Z}+\pair{AX,V}\pair{AY,Z},
$$
for $X,Y,Z,V\in T\Sigma$, where $\R$ and $\overline{\R}$ are the curvature tensors of $\Sigma^n$ and $M^{n+1}$, respectively.
Then, if $\{X, Y\}$ is an orthonormal basis for an arbitrary 2-plane tangent to $\Sigma$, we have
\begin{align}
\label{luis.7}
\nonumber K_{\Sigma}(X,Y)=&\overline{K}(X,Y)+\pair{AX,X}\pair{AY,Y}-\pair{AX,Y}^2\\
\geq & \overline{K}(X,Y)-\norm{AX}\norm{AY}-\norm{AX}^2\\
\nonumber \geq & \overline{K}(X,Y)-2\|A\|^2,
\end{align}
where the last inequality follows from the fact that
$$
\norm{AX}^2\leq \text{Tr}(A^2)\norm{X}^2=\|A\|^2
$$
for every unit vector $X$ tangent to $\Sigma$. Since we are assuming that $\sup_{\Sigma}\|A\|^2<+\infty$, it suffices to have
$\overline{K}(X,Y)$ bounded from below.

The curvature tensor of $M^{n+1}$ expressed in terms of
the curvature tensor of $\p^n$ is
\begin{align*}
\overline{\R}(U,V)W=&\R_{\p}(U^*,V^*)W^*-\mathcal{H}^2(\pi_{I})(\pair{V,W}U-\pair{U,W}V)\\
&+\mathcal{H}'(\pi_{I})\pair{W,T}(\pair{U,T}V-\pair{V,T}U)\\&-\mathcal{H}'(\pi_{I})(\pair{V,W}\pair{U,T}-\pair{U,W}\pair{V,T})T,
\end{align*}
for every $U,V,W\in TM$, where $T=\partial_t$ and we are using the notation $U^*$ to denote ${\pi_{\p}}_*U$ for an arbitrary
$U\in TM$. Then, for the orthonormal basis $\{X,Y\}$ we find that
\begin{eqnarray}
\label{luis.8}
\nonumber \overline{K}(X,Y) & = & \frac{1}{\rho^2(h)} K_{\p}(X^*,Y^*)
\norm{X^*\wedge Y^*}^2\\
{} & {} &-\mathcal{H}^2(h)-\mathcal{H}'(h)(\pair{X,\nabla h}^2+\pair{Y,\nabla h}^2)\\
\nonumber {} & \geq & \frac{1}{\rho^2(h)} K_{\p}(X^*,Y^*)\norm{X^*\wedge Y^*}^2
-\mathcal{H}^2(h)-|\mathcal{H}'(h)|,
\end{eqnarray}
since
\[
\pair{X,\nabla h}^2+\pair{Y,\nabla h}^2\leq\norm{\nabla h}^2\leq 1.
\]
On the other hand,
\begin{eqnarray*}
\norm{X^*\wedge Y^*}^2 & = &
\norm{X^*}^2\norm{Y^*}^2-\pair{X^*,Y^*}^2\\
{} & = & 1-\pair{X,T}^2-\pair{Y,T}^2\leq 1.
\end{eqnarray*}
Therefore, if $K_{\p}\geq c$ for some constant $c$, we deduce
\begin{equation}
\label{luis.9}
\frac{1}{\rho^2(h)} K_{\p}(X^*,Y^*)\norm{X^*\wedge Y^*}^2
\geq -\frac{|c|}{\rho^2(h)}.
\end{equation}
Finally, since $h$ a bounded function, we conclude from \eqref{luis.7}, \eqref{luis.8} and \eqref{luis.9} that
the sectional curvature $K(X,Y)$ is bounded from below by an absolute constant.
\end{proof}

We observe that condition \eqref{sectional} has been used in the proof of Theorem
\ref{mainthm} only to guarantee that
the Omori-Yau maximum principle holds on $\Sigma$ for the Laplacian and for the semi-elliptic operator $\mathcal{L}_1$
(or $-\mathcal{L}_1$). Therefore, the theorem remains true under any other hypothesis guaranteeing this latter fact.
Thus, and as a consequence of Corollary \ref{OYproperly}, we can also state the following:
\begin{theorem}
\label{mainthmproper}
Let $\p^n$ be a complete, non-compact, Riemannian manifold whose radial sectional curvature satisfies condition \eqref{sectP}.
Let $f: \Sigma^n \ra I \times_{\rho} \p^n$ be a properly immersed hypersurface of constant positive $2$-mean curvature
$H_2$. Assume that $\sup_{\Sigma}|H_1|<+\infty$ and that $\Sigma$ is contained in a slab.
If $\mathcal{H}'(t)>0$ almost everywhere and the angle function $\Theta$ does not change sign,
then $f(\Sigma^n)$ is a slice.
\end{theorem}

As pointed out before, for the validity of Theorem \ref{mainthmproper} it suffices to show that the Omori-Yau maximum principle
holds on $\Sigma$ for the Laplacian and for the semi-elliptic operator $\mathcal{L}_1$ (or $-\mathcal{L}_1$). But this
follows directly from Corollary \ref{OYproperly}, since $\|A\|^2=\text{Tr}(A^2)=n^2H_1^2-n(n-1)H_2$ and
$$\sup_{\Sigma}\|A\|^2\leq n^2(\sup_{\Sigma}H_1)^2-n(n-1)H_2<+\infty.$$
See the remark following Corollary \ref{OYproperly}.

\section{Hypersurfaces with constant higher order mean curvature}
In this section we will extended our previous results to the case of hypersurfaces with non-zero constant $k$-mean
curvature $H_{k}$, when $k\geq 3$. To this end, we will work with the operator $L_{k-1}$, and we will assume that there
exists an elliptic point in $\Sigma$. Note that the existence of an elliptic point is always guaranteed when $\Sigma$ is compact and $\rho'\neq0$ on $\Sigma$ (see the proof of Theorem \ref{main2compact} below and Lemma 5.3 in \cite{aliascolares} in a Lorentzian ambient space). Recall from the discussion in the Preliminaries that the existence of an elliptic point implies that $H_k$ is positive, the immersion is two-sided and $H_1>0$ for the chosen
orientation. Moreover, it implies also that, for every $1\leq j\leq k-1$, the operators $L_{j}$ are elliptic or, equivalently, the
operators $P_j$ are positive definite.

In order to extend our previous results to the case of higher order mean curvatures,
we introduce a family of operators, extending $\mathcal{L}_1$.
For $2\leq k\leq n$, we define the operator
\[
\mathcal{L}_{k-1}=\mathrm{Tr}
\Big(\Big[\sum_{j=0}^{k-1}(-1)^j\frac{c_{k-1}}{c_j}\mathcal{H}(h)^{k-1-j}\Theta^{j}P_j\Big]\circ \hess \Big)=
\mathrm{Tr}(\mathcal{P}_{k-1}\circ\hess),
\]
where
\begin{equation}
\label{luis.21}
\mathcal{P}_{k-1}=\sum_{j=0}^{k-1}(-1)^j\frac{c_{k-1}}{c_j}\mathcal{H}(h)^{k-1-j}\Theta^{j}P_j.
\end{equation}
We claim that
$$
\mathcal{L}_{k-1}\sigma(h)=c_{k-1}\rho(h)(\mathcal{H}(h)^{k}+(-1)^{k-1}\Theta^{k}H_{k}).
$$
and we prove the claim by induction. We have already seen in \eqref{luis.10} that the claim is true for $k=2$. For $k\geq 3$, we
observe that
\[
\mathcal{P}_{k-1}=\frac{c_{k-1}}{c_{k-2}}\mathcal{H}(h)\mathcal{P}_{k-2}+(-1)^{k-1}\Theta^{k-1}P_{k-1}
\]
and then
\[
\mathcal{L}_{k-1}=\frac{c_{k-1}}{c_{k-2}}\mathcal{H}(h)\mathcal{L}_{k-2}+(-1)^{k-1}\Theta^{k-1}L_{k-1}
\]
Therefore, if $k\geq 3$ and we assume that the claim is true for $\mathcal{L}_{k-2}$, then using \eqref{lrsigma} we
conclude that
\begin{eqnarray*}
\mathcal{L}_{k-1}\sigma(h) & = & \frac{c_{k-1}}{c_{k-2}}\mathcal{H}(h)\mathcal{L}_{k-2}\sigma(h)+(-1)^{k-1}\Theta^{k-1}L_{k-1}\sigma(h) \\
{} & = & c_{k-1}\rho(h)(\mathcal{H}(h)^{k}+(-1)^{k-2}\mathcal{H}(h)\Theta^{k-1}H_{k-1}\\
{} & {} & +(-1)^{k-1}\mathcal{H}(h)\Theta^{k-1}H_{k-1}+(-1)^{k-1}\Theta^{k}H_{k}) \\
{} & = & c_{k-1}\rho(h)(\mathcal{H}(h)^{k}+(-1)^{k-1}\Theta^{k}H_{k}).
\end{eqnarray*}

We are now ready to give the following extension of Theorem \ref{mainthmcompact}.
\begin{theorem}
\label{thmHrcompact}
Let $f:\Sigma^n \ra I \times_{\rho} \p^n$ be a compact hypersurface with constant $k$-mean curvature $H_k$,
with $3\leq k\leq n$. Assume that there exists an elliptic point in $\Sigma$. If
$\mathcal{H}'(t)\geq 0$ and the angle function $\Theta$ does not change sign, then $\p^n$ is necessarily compact and
 $f(\Sigma^n)$ is a slice.
\end{theorem}
\begin{proof}
Choose the orientation of $\Sigma$ so that $H_1>0$. Since $\Sigma^n$ is compact, we may apply Lemma \ref{lemmamainthm}.
Let us consider first the case when $\Theta\leq 0$, so that $\mathcal{H}(h)\geq 0$. Thus, by
\eqref{luis.21} the operator $\mathcal{P}_{k-1}$ is positive semi-definite or, equivalently,
$\mathcal{L}_{k-1}$ is semi-elliptic. Reasoning as in the proof of Theorem \ref{mainthmcompact}, yields
\[
\mathcal{L}_{k-1}\sigma(h)(\pmax)=c_{k-1}\rho(h^*)(\mathcal{H}(h^*)^k-H_k)\leq 0
\]
and
\[
\mathcal{L}_{k-1}\sigma(h)(\pmin)=c_{k-1}\rho(h_*)(\mathcal{H}(h_*)^k-H_k)\geq 0,
\]
with $\pmax\in\Sigma$ and $\pmin\in\Sigma$ such that $h(\pmax)=h^*=\max_\Sigma h$ and $h(\pmin)=h_*=\min_\Sigma h$.

Then, since $\mathcal{H}(h)\geq 0$ on $\Sigma$, we obtain
$$\mathcal{H}(h_*)\geq H_k^{1/k}\geq \mathcal{H}(h^*).$$ On the other hand, by $\mathcal{H}'\geq 0$ we also have
$\mathcal{H}(h_*)\leq\mathcal{H}(h^*)$. Thus, we have the equality $\mathcal{H}(h_*)=\mathcal{H}(h^*)$ and
$\mathcal{H}(h)=H_k^{1/k}$ is constant on $\Sigma$. Therefore, by \eqref{lrsigma} and using the Garding inequality
$H_{k-1}\geq H_k^{(k-1)/k}$ (see \eqref{garding}) and the fact that $\Theta\geq -1$, we obtain
\begin{eqnarray*}
L_{k-1} \sigma(h) & = & c_{k-1}\rho(h)H_k^{1/k}(H_{k-1}+\Theta H^{(k-1)/k}_k)\\
{} & \geq & c_{k-1}\rho(h)H_k^{1/k}(H_{k-1}-H^{(k-1)/k}_k)\geq 0.
\end{eqnarray*}
That is, $L_{k-1}\sigma(h)\geq 0$ on the compact manifold $\Sigma$. Therefore, by the maximum principle applied to the
elliptic operator $L_{k-1}$ we conclude that $\sigma(h)$, and hence $h$, is constant.

Finally, in the case where $\Theta\geq 0$ we know from Lemma \ref{lemmamainthm} that $\mathcal{H}(h)\leq 0$ on
$\Sigma$, so that the operator $(-1)^{k-1}\mathcal{L}_{k-1}$ is semi-elliptic. The proof then follows as in the case
$\Theta\leq 0$, working with $(-1)^{k-1}\mathcal{L}_{k-1}$ instead.
\end{proof}

For the case of complete (and non-compact) hypersurfaces, we can state the following extension of
Theorem \ref{mainthm}.
\begin{theorem}
\label{thmHr}
Let $f: \Sigma^n \ra I \times_{\rho} \p^n$ be a complete hypersurface with constant $k$-mean curvature $H_k$,
$3\leq k\leq n$, which satisfies condition \eqref{sectional}. Assume that there exists an elliptic point in $\Sigma$,
$\sup_{\Sigma}|H_1|<+\infty$ and $\Sigma$ is contained in a
slab. If $\mathcal{H}'(t)>0$ almost everywhere and the angle function $\Theta$ does not change sign,
then $f(\Sigma^n)$ is a slice.
\end{theorem}
\begin{proof}

By Corollary \ref{coroOY}, we know that the Omori-Yau maximum principle holds on $\Sigma$ for the Laplacian operator, so that we may apply Lemma
\ref{lemmamainthm}. Thus, in the case where $\Theta\leq 0$ we have $\mathcal{H}(h)\geq 0$ and therefore
\[
(-1)^j\mathcal{H}(h)^{k-1-j}\Theta^{j}\geq 0
\]
for every $j=0,\ldots,k-1$. Since the operators $P_0=I, P_1, \ldots, P_{k-1}$ are all positive definite, it follows from
here that the operator $\mathcal{P}_{k-1}$  is positive semi-definite or, in other words, that the differential
operator $\mathcal{L}_{k-1}$ is semi-elliptic. Furthermore, since $0\leq -\Theta\leq 1$,
\[
\mathrm{Tr}(\mathcal{P}_{k-1})=
c_{k-1}\sum_{j=0}^{k-1}(-1)^j\mathcal{H}(h)^{k-1-j}\Theta^{j}H_j\leq
c_{k-1}\sum_{j=0}^{k-1}\mathcal{H}(h^*)^{k-1-j}H^{*}_j,
\]
where $h^*=\sup_\Sigma h<+\infty$ and $H_j^{*}=\sup_\Sigma H_j\leq(\sup_\Sigma H_1)^{j}<+\infty$ because of
\eqref{garding}. Hence by Corollary \ref{coroOY}, the Omori-Yau maximum principle holds for the operator
$\mathcal{L}_{k-1}$ and, proceeding as in the proof of Theorem \ref{mainthm}, we may find two sequences
$\{p_j\}_{j\in\enne}\subset\Sigma$ and $\{q_j\}_{j\in\enne}\subset\Sigma$ satisfying
\[
\lim_{j\rightarrow+\infty}h(p_j)=h^{*}, \quad \text{and} \quad \lim_{j\rightarrow+\infty}h(q_j)=h_{*},
\]
\[
\lim_{j\rightarrow+\infty}\Theta(p_j)=\lim_{j\rightarrow+\infty}\Theta(q_j)=-1,
\]
\[
\mathcal{L}_{k-1}\sigma(h)(p_j)=
c_{k-1}\rho(h(p_j))(\mathcal{H}(h(p_j))^{k}+(-1)^{k-1}\Theta^{k}(p_j)H_{k})<\frac{1}{j},
\]
and
\[
\mathcal{L}_{k-1}\sigma(h)(q_j)=
c_{k-1}\rho(h(q_j))(\mathcal{H}(h(q_j))^{k}+(-1)^{k-1}\Theta^{k}(q_j)H_{k})<\frac{1}{j}.
\]
Making $j\rightarrow+\infty$ in the inequalties above, we obtain that
\[
\mathcal{H}(h^*)^k\leq H_k\leq \mathcal{H}(h_*)^k,
\]
which implies that $h_*=h^*$, as in the proof of Theorem \ref{mainthm}.

Finally, in the case where $\Theta\geq 0$ we proceed again as in the proof of Theorem \ref{mainthm}, working
now with the operator $(-1)^{k-1}\mathcal{L}_{k-1}$, which in this case is semi-elliptic and with
$\mathrm{Tr}((-1)^{k-1}\mathcal{P}_{k-1})$ bounded from above.
\end{proof}

As in the previous section for Theorem \ref{mainthmproper}, Theorem \ref{thmHr} remains true if we replace condition \eqref{sectional} by the stronger
condition of $\Sigma^n$ having radial sectional curvature bounded from below by a constant. By applying Lemma
\ref{lemmacurvature}, we see that this happens when the sectional curvature of $\p^n$ is itself bounded from
below.
\begin{corollary}
Let $\p^n$ be a complete Riemannian manifold with sectional curvature bounded from below and let
$f: \Sigma^n \ra I \times_{\rho} \p^n$ be a complete hypersurface with constant $k$-mean curvature $H_k$,
$3\leq k\leq n$. Assume that there exists an elliptic point in $\Sigma$,
$\sup_{\Sigma}|H_1|<+\infty$ and $\Sigma$ is contained in a
slab. If $\mathcal{H}'(t)>0$ almost everywhere and the angle function $\Theta$ does not change sign,
then $f(\Sigma^n)$ is a slice.
\end{corollary}
Indeed, by \eqref{garding} we know that $H_2>0$, so that
$\sup_{\Sigma}\|A\|^2\leq n^2(\sup_{\Sigma}H_1)^2<+\infty$ and we may apply Lemma \ref{lemmamainthm} to conclude
that the radial sectional curvature of $\Sigma$ is bounded from below. The result then follows from Theorem \ref{thmHr}.

Finally, similarly to what happened in the previous section, condition \eqref{sectional} has been used in the proof of Theorem
\ref{thmHr} only to guarantee that
the Omori-Yau maximum principle holds on $\Sigma$ for the Laplacian and for the semi-elliptic operator $\mathcal{L}_{k-1}$
(or $-\mathcal{L}_{k-1}$). Therefore, the theorem remains true under any other hypothesis guaranteing that property.
Then, and as a consequence of Corollary \ref{OYproperly}, we can also state the following:
\begin{theorem}
\label{thmHrproper}
Let $\p^n$ be a complete, non-compact, Riemannian manifold whose radial sectional curvature satisfies
condition \eqref{sectP}.
Let $f: \Sigma^n \ra I \times_{\rho} \p^n$ be a properly immersed hypersurface of constant $k$-mean curvature,
$3\leq k\leq n$. Assume that there exists an elliptic point in $\Sigma$,
$\sup_{\Sigma}|H_1|<+\infty$ and that $\Sigma$ is contained in a slab.
If $\mathcal{H}'(t)>0$ almost everywhere and the angle function $\Theta$ does not change sign,
then $f(\Sigma^n)$ is a slice.
\end{theorem}

\section{Further results for hypersurfaces with constant higher order mean curvatures}
In this last section we introduce some further results for the case of constant higher order mean curvatures. As a first result in this direction we shall prove the next
\begin{theorem}\label{main2compact}
Let $f:\Sigma^n \ra M^{n+1}=I \times_{\rho} \p^n$ be a compact hypersurface of constant $k$-mean curvature, $2\leq k \leq n$ and suppose that $\mathcal{H}$ does not vanish. Assume that
\begin{equation}\label{seccurv}
K_{\p^n} \geq \sup_I \{{\rho'}^2-\rho''\rho\},
\end{equation}
$K_{\p^n}$ being the sectional curvature of $\p^n$, and that the angle function $\Theta$ does not change sign. Then either $f(\Sigma^n)$ is a slice over a compact $\p^n$ or $M^{n+1}$ has constant sectional curvature and $\Sigma^n$ is a geodesic hypersphere. The latter case cannot occur if the inequality \eqref{seccurv} is strict.
\end{theorem}

First we proceed with the proof of two important lemmas that we will be essential in the proof of Theorem \ref{main2compact}. For Lemma \ref{divpr} see also Lemma 3.1 in \cite{aliasliramalacarne}.
\begin{lemma}\label{divpr}
Let $\Sigma^{n} \ra \overline{M}^{n+1}$ be an isometric immersion. Let $E_1,...,E_n$ be a local orthonormal frame on $\Sigma$ and $N$ be local unit normal. Then
\begin{equation}\label{divPk}
\sum_{i=1}^n\pair{(\nabla_{E_i} P_k)X,E_i}=\sum_{j=0}^{k-1}\sum_{i=1}^n(-1)^{k-1-j}\pair{\overline{\R}(E_i,A^{k-1-j}X)N,P_jE_i},
\end{equation}
for every vector field $X \in T\Sigma$.
\end{lemma}
\begin{proof}
We will prove Equation \eqref{divPk} by induction on $k$, $1\leq k\leq n-1$.
It is not difficult to prove that this is true for $k=1$ using Codazzi equation and the definition of $P_1$.
Assume that the equation holds for $k-1$. Then, using again Codazzi equation we get
\begin{align*}
\sum_{i=1}^n\pair{(\nabla_{E_i}P_k)X,E_i}=&-\sum_{i=1}^n\pair{(\nabla_{E_i} P_{k-1})AX,E_i}+\sum_{i=1}^n\pair{\overline{\R}(E_i,X)N,P_{k-1}E_i}\\
=&-\sum_{j=0}^{k-2}\sum_{i=1}^n(-1)^{k-2-j}\pair{\overline{\R}(E_i,A^{k-1-j}X)N,P_jE_i}\\
&+\sum_{i=1}^n\pair{\overline{\R}(E_i,X)N,P_{k-1}E_i}\\
=&\sum_{j=0}^{k-1}\sum_{i=1}^n(-1)^{k-1-j}\pair{\overline{\R}(E_i,A^{k-1-j}X)N,P_jE_i}.
\end{align*}
For further details see \cite{aliasliramalacarne}, paying attention to the different convention for the sign of $\overline{\R}$.
\end{proof}
\begin{corollary}
Let $\Sigma^{n} \ra I\times_{\rho} \p^n$ be an immersed hypersurface and assume that $\p^n$ has constant sectional curvature $\kappa$. Then
\begin{equation}\label{divPksecconst}
\mathrm{div}P_k=-(n-k)\Theta\Big(\frac{\kappa}{\rho^2(h)}+\mathcal{H}'(h)\Big)P_{k-1}\nabla h.
\end{equation}
\end{corollary}
\begin{proof}
Let $E_1,...,E_n$ be a local orthonormal frame on $\Sigma^n$ and observe that
\begin{equation*}
\pair{\mathrm{div}P_k,X}=\sum_{i=1}^n\pair{(\nabla_{E_i}P_k)X,E_i}
\end{equation*}
for every vector field $X \in T\Sigma$.
Fix $j$, $j=0,...,k-1$. Then
\begin{align*}
\sum_{i=1}^n\pair{\overline{\R}(E_i,A^{k-1-j}X)N,P_jE_i}=&\sum_{i=1}^n\pair{\R_{\p}({\pi_{\p}}_*E_i,{\pi_{\p}}_*A^{k-1-j}X){\pi_{\p}}_*N,P_jE_i}\\
&+\Theta\mathcal{H}'(h)(\pair{P_j\nabla h,A^{k-1-j}X}\\
&-c_jH_j\pair{\nabla h,A^{k-1-j}X}).
\end{align*}
Since $\p^n$ has constant sectional curvature $\kappa$ it follows that
$$
\R_{\p}(Y,Z)W=\kappa(\pair{Z,W}_{\p}Y-\pair{Y,W}_{\p}Z).
$$
Hence a direct calculation shows that
\begin{align*}
\sum_{i=1}^n\pair{R_{\p}({\pi_{\p}}_*E_i,{\pi_{\p}}_*A^{k-1-j}X){\pi_{\p}}_*N,P_jE_i}=&\frac{\kappa}{\rho^2(h)}\Theta(\pair{P_j\nabla h,A^{k-1-j}X}\\
&-c_jH_j\pair{\nabla h,A^{k-1-j}X}).
\end{align*}
We claim that
$$
B_k:=\sum_{j=0}^{k-1}(-1)^{k-j-1}(P_jA^{k-1-j}-c_jH_jA^{k-1-j})=-(n-k)P_{k-1}.
$$
In this case the conclusion of the Corollary is immediate.
We prove the claim by induction on $k$, $k=0,...,n-1$.
The case $k=0$ is trivial. Assume that we proved the equation for $k-1$.
Then
\begin{align*}
B_k=&P_{k-1}-c_{k-1}H_{k-1}I-B_{k-1}\circ A\\
=&P_{k-1}-c_{k-1}H_{k-1}I+(n-k+1)P_{k-2}A\\
=&-(n-k)P_{k-1}.
\end{align*}

\end{proof}
\begin{lemma}
Let $\Sigma^n$ be a hypersurface immersed into a warped product space $I\times_{\rho}\p^n$, with angle function $\Theta$ and height function $h$. Let $\hat{\Theta}=\rho \Theta$. Then, for every $k=0,...,n-1$ we have
\begin{align*}
L_k \hat{\Theta}=&-{n \choose {k+1}}\rho(h)\pair{\nabla h, \nabla H_{k+1}}-\rho'(h)c_k H_{k+1}\\
&-\hat{\Theta}\mathcal{H}'(h)(\norm{\nabla h}^2c_kH_k-\pair{P_k \nabla h,\nabla h})-\frac{\hat{\Theta}}{\rho(h)^2}\beta_k\\
&-\hat{\Theta}{n \choose{k+1}} (nH_1H_{k+1}-(n-k-1)H_{k+2}),
\end{align*}
where
$$
\beta_k=\sum_{i=1}^n\mu_{k,i}K_{\p}({\pi_{\p}}_*E_i,{\pi_{\p}}_*N)\norm{{\pi_{\p}}_*E_i\wedge {\pi_{\p}}_*N}^2.
$$
Here the $\mu_{k,i}$'s stand for the eigenvalues of $P_k$ and $\{E_1,\ldots,E_n\}$ is a local orthonormal frame on $\Sigma$ diagonalizing $A$.
\end{lemma}
\begin{proof}
Since $\rho(t)T$ is a conformal vector field
$$
\nabla \hat{\Theta}=-\rho(h)A\nabla h.
$$
Therefore, using Equation \eqref{hessh} we find
$$
\nabla_X \nabla \hat{\Theta}= -\rho(h)(\nabla_X A)\nabla h-\rho'(h)AX-\hat{\Theta}A^2X.
$$
Hence
\begin{align*}
L_k \hat{\Theta}=& -\rho(h)\sum_{i=1}^n\pair{P_k(\nabla_{E_i} A)\nabla h,E_i}\\
&-\rho'(h)c_kH_{k+1}-{n \choose {k+1}}\hat{\Theta}(H_1H_{k+1}-(n-k-1)H_{k+2}).
\end{align*}
Using the expression of the covariant derivative of a tensor field we get
\begin{align*}
-P_k(\nabla_{E_i} A)\nabla h=&(\nabla_{E_i} P_k)A\nabla h-(\nabla_{E_i}P_kA)\nabla h\\
=&(\nabla_{E_i} P_k)A\nabla h+(\nabla_{E_i} P_{k+1})\nabla h-E_i(S_{k+1})\nabla h.
\end{align*}
By equation \eqref{divPk} it follows that
\begin{align*}
-\sum_{i=1}^n\pair{P_k(\nabla_{E_i} A)\nabla h,E_i}=&\sum_{i=1}^{n}\pair{(\nabla_{E_i} P_k)A\nabla h,E_i}\\
&+\sum_{i=1}^n\pair{(\nabla_{E_i} P_{k+1})\nabla h,E_i}-\nabla h(S_{k+1})\\
=&\sum_{i=1}^n\pair{\overline{\R}(E_i,\nabla h)N,P_k E_i}-\nabla h(S_{k+1}).
\end{align*}
Since $\nabla h=T-\Theta N$, we can write
$$
\overline{\R}(E_i,\nabla h)N=\overline{\R}(E_i,T)N-\Theta\overline{\R}(E_i,N)N.
$$
Using Gauss equation and observing that ${\pi_{\p}}_*T=0$ we get
$$
\overline{\R}(E_i,T)N=-(\mathcal{H}(h)^2+\mathcal{H}'(h))\Theta E_i=-\frac{\rho''(h)}{\rho(h)}\Theta E_i.
$$
So
$$
\sum_{i=1}^n\pair{\overline{\R}(E_i,T)N,P_kE_i}=-\frac{\rho''(h)}{\rho(h)}\Theta c_kH_k.
$$
Again by Gauss equation
\begin{align*}
\overline{\R}(E_i,N)N=&\R_{\p}({\pi_{\p}}_*E_i,{\pi_{\p}}_*N){\pi_{\p}}_*N-\mathcal{H}(h)^2E_i\\
&+\mathcal{H}'(h)\Theta(\pair{E_i,\nabla h}N-\Theta E_i)-\mathcal{H}'(h)\pair{E_i,\nabla h}T.
\end{align*}
Assume that the orthonormal basis $\{E_i\}_1^n$ diagonalizes $A$ and hence $P_k$, that is $P_kE_i=\mu_{k,i}E_i$. Then
\begin{align*}
\sum_{i=1}^n\pair{\overline{\R}(E_i,N)N,P_kE_i}=&\frac{1}{\rho(h)^2}\sum_{i=1}^n\mu_{k,i}K_{\p}({\pi_{\p}}_*E_i,{\pi_{\p}}_*N)\norm{{\pi_{\p}}_*E_i\wedge {\pi_{\p}}_*N}^2\\
&-\frac{\rho''(h)}{\rho(h)}c_kH_k+\mathcal{H}'(h)(\norm{\nabla h}^2c_kH_k-\pair{P_k \nabla h,\nabla h}).
\end{align*}
Thus,
\begin{align*}
\sum_{i=1}^n\pair{\overline{\R}(E_i,\nabla h)N,P_k E_i}=&\sum_{i=1}^n\pair{\overline{\R}(E_i,T)N,P_kE_i}-\Theta\sum_{i=1}^n\pair{\overline{\R}(E_i,N)N,P_kE_i}\\
=&-\frac{\Theta}{\rho(h)^2}\sum_{i=1}^n\mu_{k,i}K_{\p}({\pi_{\p}}_*E_i,{\pi_{\p}}_*N)\norm{{\pi_{\p}}_*E_i\wedge{\pi_{\p}}_*N}^2\\
-&\Theta\mathcal{H}'(h)(\norm{\nabla h}^2c_kH_k-\pair{P_k \nabla h,\nabla h})
\end{align*}
and this concludes the proof of the lemma.
\end{proof}
\begin{corollary}\label{lrthetasecconst}
Let $\Sigma^n$ be a hypersurface immersed into a warped product space $I\times_{\rho}\p^n$, with angle function $\Theta$ and height function $h$. Assume that $\p^n$ has constant sectional curvature $\kappa$ and let $\hat{\Theta}=\rho(h) \Theta$. Then, for every $k=0,...,n-1$ we have
\begin{align*}
L_k \hat{\Theta}=&-{n \choose {k+1}}\rho(h)\pair{\nabla h, \nabla H_{k+1}}-\rho'(h)c_k H_{k+1}\\
&-\hat{\Theta}\Big(\frac{\kappa}{\rho^2(h)}+\mathcal{H}'(h)\Big)(\norm{\nabla h}^2c_kH_k-\pair{P_k \nabla h,\nabla h})\\&-\hat{\Theta}{n \choose{k+1}} (nH_1H_{k+1}-(n-k-1)H_{k+2}).
\end{align*}
\end{corollary}
We are now ready to give the
\begin{proof}[Proof of Theorem \ref{main2compact}]
We may assume without loss of generality that $\mathcal{H}(h)>0$ on $\Sigma$. Since $\Sigma^n$ is compact, there exists a point $p_0 \in \Sigma$ where the height function attains its maximum. Then $\nabla h(p_0)=0$, $\Theta(p_0)=\pm 1$ and by \eqref{hessh}
$$
\Hess h(p_0)(v,v)=\mathcal{H}(h^*)\pair{v,v}+\Theta(p_0)\pair{Av,v}(p_0)\leq 0.
$$
If $\Theta(p_0)=-1$, then
$$
\pair{Av,v}(p_0)\geq\mathcal{H}(h^*)\pair{v,v}>0,
$$
for any $v\neq0$.
Thus $p_0$ is an elliptic point, $H_k$ is a positive constant and by Garding inequalities
$$
H_1 \geq H_2^{\frac{1}{2}} \geq \cdots \geq H_k^{\frac 1k}>0$$
with equality only at umbilical points. In particular, $\Sigma$ is two-sided and then $\Theta\leq0$. If $\Theta(p_0)=1$, changing the orientation we have the same conclusion.

Consider the function
$$
\phi=\sigma(h)H_k^{\frac 1k}+\rho(h)\Theta.
$$
Let us prove that $L_{k-1} \phi \geq 0$.
Since $H_k$ is constant we have
\begin{align*}
L_{k-1}\phi=&H_k^{\frac 1k}L_{k-1} \sigma(h)+L_{k-1} \hat{\Theta} \\
=& c_{k-1}H_k^{\frac 1k}(\rho'(h)H_{k-1}+\hat{\Theta}H_k)-c_{k-1}H_{k-1}\hat{\Theta}\mathcal{H}'(h)\norm{\nabla h}^2\\
&+ \hat{\Theta}\mathcal{H}'(h)\pair{P_{k-1} \nabla h, \nabla h}-\hat{\Theta}{n \choose k} (nH_1H_k-(n-k)H_{k+1})\\
&-\rho'(h)c_{k-1} H_k- \frac{\hat{\Theta}}{\rho(h)^2}\sum_{i=1}^n\mu_{k-1,i}K_{\p}({\pi_{\p}}_*E_i,{\pi_{\p}}_*N)\norm{{\pi_{\p}}_*E_i\wedge{\pi_{\p}}_*N}^2\\
=&A+B+C,
\end{align*}
where
$$
A=-\hat{\Theta}{n \choose k} (nH_1H_k-(n-k)H_{k+1}-kH_k^{\frac{k+1}{k}}),
$$
$$
B=c_{k-1}\rho'(h)(H_{k-1}H_k^{\frac 1k}-H_k)
$$
and
\begin{align*}
C=&-\hat{\Theta}\mathcal{H}'(h)(\norm{\nabla h}^2c_{k-1}H_{k-1}-\pair{P_{k-1} \nabla h,\nabla h})\\
&-\frac{\hat{\Theta}}{\rho(h)^2}\sum_{i=1}^n\mu_{k-1,i}K_{\p}({\pi_{\p}}_*E_i,{\pi_{\p}}_*N)\norm{{\pi_{\p}}_*E_i\wedge{\pi_{\p}}_*N}^2.
\end{align*}
Then by Garding inequalities,
$$
H_{k-1}H_k^{\frac 1k}-H_k=H_k^{\frac 1k}(H_{k-1}-H_k^{\frac{k-1}{k}}) \geq 0.
$$
Moreover,
$$nH_1H_k-kH_k^{\frac{k+1}{k}} \geq nH_k^{\frac{k+1}{k}}-kH_k^{\frac{k+1}{k}}=(n-k)H_k^{\frac{k+1}{k}},
$$
hence
$$
nH_1H_k-kH_k^{\frac{k+1}{k}}-(n-k)H_{k+1} \geq (n-k)(H_k^{\frac{k+1}{k}}-H_{k+1})\geq 0.
$$
Finally, let $\alpha:=\sup_I \{{\rho'}^2-\rho''\rho\}$. Since
$$
\norm{{\pi_{\p}}_*E_i\wedge{\pi_{\p}}_*N}^2=\norm{\nabla h}^2-\pair{E_i,\nabla h}^2,
$$
taking into account that the $\mu_{k-1,i}$'s are positive, we have
\begin{align*}
&\sum_{i=1}^n\mu_{k-1,i}K_{\p}({\pi_{\p}}_*E_i,{\pi_{\p}}_*N)\norm{{\pi_{\p}}_*E_i\wedge{\pi_{\p}}_*N}^2\\
& \geq \alpha \sum_{i=1}^n\mu_{k-1,i}\norm{{\pi_{\p}}_*E_i\wedge{\pi_{\p}}_*N}^2\\
&=\alpha(c_{k-1}H_{k-1}\norm{\nabla h}^2-\pair{P_{k-1}\nabla h,\nabla h}).
\end{align*}
Hence,
\begin{align*}
\frac{1}{\rho(h)^2}&\sum_{i=1}^n\mu_{k-1,i}K_{\p}({\pi_{\p}}_*E_i,{\pi_{\p}}_*N)\norm{{\pi_{\p}}_*E_i\wedge{\pi_{\p}}_*N}^2\\
&+\mathcal{H}'(h)(\norm{\nabla h}^2c_{k-1}H_{k-1}-\pair{P_{k-1} \nabla h,\nabla h})\\
\geq & \Big(\frac{\alpha}{\rho(h)^2}+\mathcal{H}'(h)\Big)(\norm{\nabla h}^2c_{k-1}H_{k-1}-\pair{P_{k-1} \nabla h,\nabla h})
\geq 0,
\end{align*}
where the last inequality follows from $\alpha=\sup_I\{-\rho^2\mathcal{H}'\}$ and from the fact that $P_{k-1}$ is a positive definite operator.\\
Thus, $L_{k-1}\phi \geq 0$. Since $L_{k-1}$ is an elliptic operator and $\Sigma$ is compact, we conclude by the maximum principle that $\phi$ must be constant. Hence $L_{k-1} \phi=0$ and the three terms $A$, $B$ and $C$ in $L_{k-1} \phi$ vanish on $\Sigma$.

In particular $B=0$ implies that $\Sigma$ is a totally umbilical hypersurface. Moreover, since $H_k$ is a positive constant and $\Sigma$ is totally umbilical, all the higher order mean curvatures are constant. In particular, $H_1$ is constant and the conclusion follows by Theorem 3.4 \cite{aliasdajczer}.
\end{proof}
On the other hand, using Theorem 3.1 in \cite{aliasdajczer} we can give the following version of Theorem 3.4 in \cite{aliasdajczer} for the complete case.
\begin{prop}
Let $M^{n+1}=I \times_{\rho} \p^n$ be a warped product space and assume that the Ricci curvature of $\p^n$ satisfies
\begin{equation}\label{riccicurv}
\ricc_{\p} > \sup_I \{{\rho'}^2-\rho''\rho\}.
\end{equation}
Let $f:\Sigma^n \ra I \times_{\rho} \p^n$ be a  complete, parabolic, two-sided hypersurface with constant mean curvature. Suppose that
$$
f(\Sigma^n)\subset [t_1,t_2] \times \p^n,
$$
where $t_1$, $t_2$ are finite. If the angle function $\Theta$ does not change sign, then $f(\Sigma^n)$ is a slice.
\end{prop}
For the proof, observe that since the function $\phi=H\sigma(h)+\hat{\Theta}$ is subharmonic and $\Sigma$ is parabolic, then it must be constant. In particular, $\Delta \phi=0$ and by Equation (3.8) in \cite{aliasdajczer} we conclude that $h$ has to be constant, because of the strict inequality in \eqref{riccicurv}.

In what follows, we extend this result to higher order mean curvatures. Towards this aim, we let $\mathfrak{L}_k$ be the operator
$$
\mathfrak{L}_k f=\mathrm{div}(P_k \nabla f),
$$
where $f \in C^\infty(\Sigma)$.
Notice that
$$
\mathfrak{L}_k f=\pair{\mathrm{div}P_k,\nabla f}+L_k f.
$$
We introduce the following
\begin{definition}
We will say that the manifold $\Sigma^n\hookrightarrow I \times_{\rho} \p^n$ is $\mathfrak{L}_k$-parabolic if the only bounded above $C^1$ solutions of the inequality
$$
\mathfrak{L}_k f \geq 0
$$
are constant.
\end{definition}
The following theorem is a special case of Theorem 2.6 in \cite{pirise2}
\begin{theorem}
Let $\Sigma^n\hookrightarrow I \times_{\rho} \p^n$ be a complete manifold. Fix an origin $o\in\Sigma$. If
\begin{equation}\label{eqrpar}
\Big(\sup_{\partial B_t} H_{k-1} \rm{vol}(\partial B_t)\Big)^{-1}\notin L^1(+\infty),
\end{equation}
where $\partial B_t$ is the geodesic sphere of radius $t$ centered at $o$, then $\Sigma^n$ is $\mathfrak{L}_{k-1}$-parabolic.
\end{theorem}
We are ready to state our last result.
\begin{theorem}\label{main2complete}
Let $M^{n+1}=I \times_{\rho} \p^n$ be a warped product space and assume that $\p^n$ has constant sectional curvature $\kappa$ satisfying
\begin{equation}\label{constseccurv}
\kappa > \sup_I \{{\rho'}^2-\rho''\rho\}.
\end{equation}
Let $f:\Sigma^n \ra I \times_{\rho} \p^n$ be a  complete hypersurface with $\sup_\Sigma |H_1|<+\infty$ and satisfying condition \eqref{eqrpar}. Suppose that $f$ has constant $k$-mean curvature, $2\leq k\leq n$, and
$$
f(\Sigma^n)\subset [t_1,t_2] \times \p^n,
$$
where $t_1$, $t_2$ are finite. Assume that either $k=2$ and $H_2$ is positive or $k\geq3$ and there exists an elliptic point $p \in \Sigma^n$. If $\mathcal{H}(h)$ and the angle function $\Theta$ do not change sign, then $f(\Sigma^n)$ is a slice.
\end{theorem}
\begin{remark}
\rm{Comparing with Theorem \ref{main2compact} we have relaxed the condition on $\mathcal{H}$ but we are requiring, as it will be clear from the proof, the existence of an elliptic point. That, on a compact manifold was guaranteed by the assumption $\mathcal{H}>0$. Moreover, we observe that the angle function is indeed well defined because $\Sigma$ is two-sided. For $k=2$, this follows from the positivity of $H_2$ since $H_1^2\geq H_2>0$. In the remaining cases this property follows from Garding inequalities, as in the compact case. In any case we choose the orientation so that $H_1>0$.
}
\end{remark}

\begin{proof}
It follows from the hypotheses that $\sup_{\Sigma}\norm{A}<+\infty$ and therefore by Lemma \ref{lemmacurvature} the sectional curvature of $\Sigma$ is bounded from below. We deduce then the validity of the Omori-Yau maximum principle for the Laplacian.
Assume $\mathcal{H}(h)\geq0$. Applying the Omori-Yau maximum principle to the Laplace operator and using Equation \eqref{lrh} we find that
$$
-\sgn\Theta\liminf_{j \ra +\infty}H_1(q_j) \geq \mathcal{H}(h^*)\geq 0.
$$
Therefore for the chosen orientation, $\sgn\Theta=-1$ and  $\Theta\leq0$ on $\Sigma$. Consider the operator
$$
\mathfrak{L}_{k-1} f=\mathrm{div}(P_{k-1} \nabla f)
$$
and the function
$$
\phi=H_k^{\frac 1k}\sigma(h)+\hat{\Theta},
$$
where $\hat{\Theta}=\rho(h)\Theta$.
Since $\p^n$ has constant sectional curvature $\kappa$, it follows by Equation \eqref{divPksecconst} that
\begin{align*}
\mathfrak{L}_{k-1} \phi=&-(n-k+1)\Theta\Big(\frac{\kappa}{\rho^2(h)}+\mathcal{H}'(h)\Big)\pair{P_{k-2}\nabla h,\nabla \phi}+L_{k-1} \phi \\
=&-(n-k+1)\hat{\Theta}\Big(\frac{\kappa}{\rho^2(h)}+\mathcal{H}'(h)\Big)\pair{P_{k-2}\nabla h,\nabla h}H_k^{\frac1k}\\&+(n-k+1)\hat{\Theta}\Big(\frac{\kappa}{\rho^2(h)}+\mathcal{H}'(h)\Big)\pair{P_{k-2}A\nabla h,\nabla h}\\&+H_k^{\frac1k}L_{k-1} \sigma(h)+L_{k-1} \hat{\Theta}.
\end{align*}
Using Equation \eqref{lrsigma} and Corollary \ref{lrthetasecconst} we find
\begin{align}\label{eqmarco1}
\mathfrak{L}_{k-1} \phi=&c_{k-1}\rho'(h)H_k^{\frac1k}(H_{k-1}-H_k^{\frac{k-1}{k}})\nonumber\\
&-{n \choose k}\hat{\Theta}\left(nH_1H_k-(n-k)H_{k+1}-kH_k^{\frac{k+1}{k}}\right)\nonumber\\
&-(n-k)\hat{\Theta}\Big(\frac{\kappa}{\rho^2(h)}+\mathcal{H}'(h)\Big)\pair{P_{k-1}\nabla h,\nabla h}\\
&-(n-k+1)\hat{\Theta}H_k^{\frac 1k}\Big(\frac{\kappa}{\rho^2(h)}+\mathcal{H}'(h)\Big)\pair{P_{k-2} \nabla h,\nabla h}\nonumber.
\end{align}
Using Garding inequalities as in Theorem \ref{main2compact}, it is easy to prove that the first and the second terms are nonnegative. By the fact that each $P_j$ is an elliptic operator, $j=0,...,k-1$, and by Equation \eqref{seccurv}, it follows that also all the remaining terms in the previous equation are nonnegative. Thus $\mathfrak{L}_{k-1} \phi \geq 0$. Since, by assumption \eqref{eqrpar} $\Sigma^n$ is $\mathfrak{L}_{k-1}$-parabolic, we conclude that $\phi$ has to be constant. In particular, $\mathfrak{L}_{k-1} \phi=0$ and the four terms on the right-hand side of Equation \eqref{eqmarco1} vanish. Let us prove that $\mathcal{U}=\{p \in \Sigma^n : \Theta(p)=0\}$ has empty interior. Indeed, assume the contrary and let $\mathcal{V} \neq \emptyset$ be an open subset of $\mathcal{U}$. On $\mathcal{V}$ the function $\phi=\sigma(h)H_k^{1/k}$ is constant. Hence, since $H_k \neq 0$, then $\sigma(h)$ and, equivalently $h$, is constant on $\mathcal{V}$, which is not possible since $\norm{\nabla h}^2=1-\Theta^2=1$ on $\mathcal{V}$. Therefore, since the third term on the right-hand of \eqref{eqmarco1} vanishes identically, we have
$$
\pair{P_{k-1}\nabla h,\nabla h}=0.
$$
Since $P_{k-1}$ is positive definite, this means that $h$ has to be constant.
\end{proof}

\bibliographystyle{amsplain}
\bibliography{biblioAIR}

\providecommand{\bysame}{\leavevmode\hbox to3em{\hrulefill}\thinspace}
\providecommand{\MR}{\relax\ifhmode\unskip\space\fi MR }
\providecommand{\MRhref}[2]{%
  \href{http://www.ams.org/mathscinet-getitem?mr=#1}{#2}
}
\providecommand{\href}[2]{#2}
\begin{thebibliography}{10}

\bibitem{al}
A.~D. Alexandrov, \emph{A characteristic property of spheres}, Ann. Mat. Pura
  Appl. (4) \textbf{58} (1962), 303--315.

\bibitem{aliascolares}
L.~J. Al{\'{\i}}as and A.~G. Colares, \emph{Uniqueness of spacelike
  hypersurfaces with constant higher order mean curvature in generalized
  {R}obertson-{W}alker spacetimes}, Math. Proc. Cambridge Philos. Soc.
  \textbf{143} (2007), no.~3, 703--729.

\bibitem{aliasdajczer2}
L.~J. Al{\'{\i}}as and M.~Dajczer, \emph{Uniqueness of constant mean curvature
  surfaces properly immersed in a slab}, Comment. Math. Helv. \textbf{81}
  (2006), no.~3, 653--663.

\bibitem{aliasdajczer}
L.~J. Al\'ias and M.~Dajczer, \emph{Constant mean curvature hypersurfaces in
  warped product spaces}, Proc. Edinb. Math. Soc. (2) \textbf{50} (2007),
  511--526.

\bibitem{aliasliramalacarne}
L.~J. Al{\'{\i}}as, J.~H.~S. de~Lira, and J.~M. Malacarne, \emph{Constant
  higher-order mean curvature hypersurfaces in {R}iemannian spaces}, J. Inst.
  Math. Jussieu \textbf{5} (2006), no.~4, 527--562.

\bibitem{aliasimperarigolilor}
L.~J. Al{\'{\i}}as, D.~Impera, and M.~Rigoli, \emph{Spacelike hypersurfaces of
  constant $k$-mean curvature in generalized robertson-walker spacetimes}, To
  appear in Math. Proc. Cambridge Philos. Soc. doi: 10.1017/S0305004111000697.

\bibitem{barbosacolares}
J.~L.~M. Barbosa and A.~G. Colares, \emph{Stability of {H}ypersurfaces with
  {C}onstant $r$-{M}ean {C}urvature}, Ann. Glob. An. Geom. \textbf{15} (1997),
  277--297.

\bibitem{elbert}
M.~F. Elbert, \emph{Constant positive 2-mean curvature hypersurfaces}, Illinois
  J. Math. \textbf{46} (2002), no.~1, 247--267.

\bibitem{Ga}
L.~G{\.a}rding, \emph{An inequality for hyperbolic polynomials}, J. Math. Mech.
  \textbf{8} (1959), 957--965.

\bibitem{GT}
D.~Gilbarg and N.~S. Trudinger, \emph{Elliptic partial differential equations
  of second order}, second ed., Grundlehren der Mathematischen Wissenschaften
  [Fundamental Principles of Mathematical Sciences], vol. 224, Springer-Verlag,
  Berlin, 1983.

\bibitem{montiel}
S.~Montiel, \emph{Unicity of constant mean curvature hypersurfaces in some
  {R}iemannian manifolds}, Indiana Univ. Math. J. \textbf{48} (1999), no.~2,
  711--748.

\bibitem{Om}
H.~Omori, \emph{Isometric immersions of {R}iemannian manifolds}, J. Math. Soc.
  Japan \textbf{19} (1967), 205--214.

\bibitem{pirise2}
S.~Pigola, M.~Rigoli, and A.~G. Setti, \emph{A {L}iouville-type result for
  quasi-linear elliptic equations on complete {R}iemannian manifolds}, J.
  Funct. Anal. \textbf{219} (2005), no.~2, 400--432.

\bibitem{pirise}
\bysame, \emph{Maximum principles on {R}iemannian manifolds and applications},
  Mem. Amer. Math. Soc. \textbf{174} (2005), no.~822, x+99.

\bibitem{Y}
S.~T. Yau, \emph{Harmonic functions on complete {R}iemannian manifolds}, Comm.
  Pure Appl. Math. \textbf{28} (1975), 201--228.

\end{thebibliography}

\end{document}